\RequirePackage{rotating}
\documentclass{mcom-l}

\usepackage{mathrsfs,bm}
\usepackage{cases}
\usepackage{booktabs,multirow}
\usepackage{algorithm,algorithmic}
\usepackage{graphicx}
\usepackage{subfigure}

\usepackage[colorlinks=true]{hyperref}
\hypersetup{urlcolor=blue, citecolor=blue,linkcolor=blue}

\newtheorem{theorem}{Theorem}[section]
\newtheorem{lemma}[theorem]{Lemma}
\theoremstyle{definition}
\newtheorem{definition}[theorem]{Definition}

\theoremstyle{remark}
\newtheorem{remark}[theorem]{Remark}
\theoremstyle{assumption}

\theoremstyle{proposition}
\newtheorem{proposition}[theorem]{Proposition}

\numberwithin{equation}{section}

\newcommand{\dom}[1]{\mathrm{\bf dom}\,{(#1)}} 
\newcommand{\intset}[1]{\mathrm{\bf int}\,{(#1)}} 
\newcommand{\prox}{\mathrm{\bf prox}} 
\newcommand{\Bprox}{\mathrm{\bf Bprox}} 

\def\nn{\nonumber}

\def\B{\mathscr{B}}
\def\R{{\mathbb R}}

\def\L{{\mathscr L}}

\def\X{\mathbb {X}}
\def\Y{\mathbb{Y}}

\def\bbf{{\bm{f}}}

\def\bx{{\bm x}}

\def\bA{{\bm{A}}}

\allowdisplaybreaks

\begin{document}
\graphicspath{{./FIG/},{./PIC/}}

\title[A Triple-Bregman Balanced Primal-Dual Algorithm]{A Triple-Bregman Balanced Primal-Dual Algorithm for Saddle Point Problems}


\author{Jintao Yu}
\address{Department of Mathematics and Statistics, Ningbo University, Ningbo, 315211, China.}
\curraddr{}
\email{yujintao0045@163.com}
\thanks{}

\author{Hongjin He}
\address{Department of Mathematics and Statistics, Ningbo University, Ningbo, 315211, China.}
\curraddr{}
\email{hehongjin@nbu.edu.cn}
\thanks{H. He was supported in part by National Natural Science Foundation of China (No. 12371303) and Zhejiang Provincial Natural Science Foundation of China (No. LZ24A010001). }
\thanks{H. He is the corresponding author.}

\date{}

\dedicatory{}

\begin{abstract}
The \textbf{p}rimal-\textbf{d}ual \textbf{h}ybrid \textbf{g}radient (PDHG) method is one of the most popular algorithms for solving saddle point problems. However, when applying the PDHG method and its many variants to some real-world models commonly encountered in signal processing, imaging sciences, and statistical learning, there often exists an imbalance between the two subproblems, with the dual subproblem typically being easier to solve than the primal one. In this paper, we propose a flexible \textbf{t}riple-\textbf{B}regman balanced primal-\textbf{d}ual \textbf{a}lgorithm (TBDA) to solve a class of (not necessarily smooth) convex-concave saddle point problems with a bilinear coupling term. Specifically, our TBDA mainly consists of two dual subproblems and one primal subproblem. Moreover, three Bregman proximal terms, each one with an individual Bregman kernel function, are embedded into the respective subproblems. In this way, it effectively enables us to strike a practical balance between the primal and dual subproblems. More interestingly, it provides us a flexible algorithmic framework to understand some existing iterative schemes and to produce customized structure-exploiting algorithms for applications. Theoretically, we first establish the global convergence and ergodic convergence rate of the TBDA under some mild conditions. In particular, our TBDA allows larger step sizes than the PDHG method under appropriate parameter settings. Then, when the requirements on objective functions are further strengthened, we accordingly introduce two improved versions with better convergence rates than the original TBDA. Some numerical experiments on synthetic and real datasets demonstrate that our TBDA performs better than the PDHG method and some other efficient variants in practice.
\end{abstract}


\maketitle

\section{Introduction}\label{sec1}
In this paper, we are concerned with a class of (not necessarily smooth) convex-concave saddle point problems, in which a bilinear term couples the primal and dual variables. Mathematically, the problem under consideration reads as
\begin{equation}\label{sdp}
\min_{x \in \mathbb{X}}\max_{y \in \mathbb{Y}} ~\L(x,y) := f(x) +\langle Ax, y \rangle  - g(y),
\end{equation}
where $\mathbb{X}\subseteq \R^{n}$ and $\mathbb{Y}\subseteq \R^{m}$ are two nonempty, closed and convex sets, $f: \mathbb{X} \rightarrow (-\infty, \infty]$ and $g: \mathbb{Y} \rightarrow (-\infty, \infty]$ are proper closed convex functions, $A: \mathbb{X} \rightarrow \mathbb{Y}$ is a bounded linear operator, $\langle \cdot, \cdot \rangle$ represents the standard inner product with induced Euclidean norm $\|\cdot\| \equiv \sqrt{\langle \cdot, \cdot \rangle}$ for vectors. Under certain conditions, \eqref{sdp} reduces to the well-studied composite optimization problem:
\begin{equation}\label{ccp}
\min_{x \in \mathbb{X}} f(x)  + g^*(Ax)
\end{equation}
where  $g^*$ corresponds to the conjugate function of $g$. Also, it provides an equivalent way to handle the following linearly constrained convex minimization problem:
\begin{equation}\label{LCOP}
\min_{x \in \mathbb{X}} \left\{f(x)\; | \; Ax=b\right\},
\end{equation}
where $b \in \R^m$ is a given vector. Therefore, although
we will restrict the later discussion to the case of \eqref{sdp} with vector variables, all of our results are
applicable to the case with matrix variables. In recent years, we have observed the successful and diverse applications of saddle point problems in the fields of game theory, signal processing, imaging sciences, machine learning, and statistical learning, e.g., see \cite{CP16,CKCH23,EZC10,KP15,RHLNSH20,Suh22,Val21,ZC08b}, to name a few.

In the literature, one of the most popular solvers for \eqref{sdp} is the first-order primal-dual algorithm \cite{CP11,CP16b}, which is also named as the primal-dual hybrid gradient (PDHG) method in the imaging community \cite{BR12,EZC10,HMY17,ZC08b}. Given the $k$-th iterate $(x^k,y^k)$, the iterative scheme of the PDHG method \cite{CP11} reads as
\begin{numcases}{\label{pdhg}}
x^{k+1} = \arg\min_{x\in \mathbb{X}}\left\{ f(x) + \langle Ax, y^{k}\rangle + \frac{\mu}{2} \| x-x^{k} \|^{2} \right\}, \nonumber\\
\tilde{x}^{k+1} = x^{k+1} + \sigma(x^{k+1}-x^{k}), \\
y^{k+1} = \arg\max_{y\in \mathbb{Y}}\left\{-g(y) + \langle A \tilde{x}^{k+1}, y\rangle - \frac{\gamma}{2} \| y-y^{k} \|^{2} \right\}, \nonumber
\end{numcases}
where  $\mu > 0$ and $\gamma > 0$ are proximal parameters serving as step sizes, and $\sigma \in [0,1]$ is an extrapolation parameter. Particularly, \eqref{pdhg} reduces to the Arrow-Hurwicz Primal-Dual (AHPD) method \cite{AHH58} when taking $\sigma=0$. Unfortunately, the AHPD method is not necessarily convergent with fixed $\mu$ and $\gamma$ under the standard convexity conditions of $f$ and $g$ \cite{HXY22}, even it also has global convergence and sublinear convergence rate under relatively strong assumptions \cite{EZC10,HXY22,HYY15,MSMC15,NO09}. Comparatively, the version with $\sigma=1$ requires weaker convergence-guaranteeing conditions than the AHPD method.
Specifically, it was documented in \cite{CP11,CP16b} that the PDHG method equipped with $\sigma = 1$ is globally convergent under generic convex conditions as long as
\begin{equation}\label{condition}
\|A^\top A \|< \mu\gamma,
\end{equation}
where $\|A A^\top \|$ is the spectrum of $AA^\top$. Actually, as reported in \cite{CHX13,HY12b}, the extrapolation parameter $\sigma$ plays a crucial role in weakening the convergence-guaranteeing condition and accelerating the PDHG method. Therefore, He and Yuan \cite{HY12b} enlarged the range of $\sigma$ as $\sigma \in [-1,1]$ with the help of an extra correction step. Later, some more improved variants of the PDHG method with relaxed step sizes and extrapolation parameter were developed in recent papers, e.g., see \cite{BCYZ25,CHX13,CY21,HMXY22,JZH23,LY24,LY21,MCJH23,WH20}. Furthermore, some researchers introduced dynamic strategy to adjust the step sizes in \cite{BR11,BR14,CT22,CYZ22,MP18}. 
As shown in \cite{HHYY14}, when $\mathbb{X}$ and/or $\mathbb{Y}$ are not the whole spaces, the $x$- and/or $y$-subproblems in \eqref{pdhg} possibly lose their closed-form solutions. In this case, it usually takes much computing time to find approximate solutions of the subproblems. Therefore, some inexact variants of the PDHG method (e.g., see \cite{JCWH21,JWCZ21,RC20,WYH23}) were introduced for the purpose of reducing the computational cost of exactly solving subproblems.

When applying the PDHG method \eqref{pdhg} and its variants to some specific models (e.g., \eqref{LCOP} and its concrete forms in Section \ref{Sec:numexp}), it is noteworthy that the dual subproblem often is easier than its primal subproblem in the sense that the dual one has lower computational cost. Interestingly, as shown in \cite{HY21}, the augmented Lagrangian method (ALM, which is also a primal-dual method) can be greatly improved by using a so-named balancing strategy to increase the computational cost of the dual subproblem (which corresponds to the Lagrangian multiplier). 
To some extend, the newly introduced balancing strategy in \cite{HY21} can be regarded as one of the ``dual stabilization'' techniques mentioned in \cite{IS15} to speed up primal-dual algorithms for solving \eqref{LCOP}. More recently, He et al. \cite{HWY23} proposed a so-called symmetric primal-dual algorithm (SPIDA) for \eqref{sdp}, where they introduced another type balancing strategy by calculating the dual subproblem twice to achieve a balance between the primal and dual subproblems. Actually, their SPIDA shares the similar idea used in the alternating extragradient method \cite{BR11,BR14} and the primal-dual fixed point algorithm \cite{CHZ16}. We can easily observe from the reported experiments in these papers that one more computation on the dual subproblem is also able to improve the performance of primal-dual algorithms on solving \eqref{sdp} (including \eqref{ccp} and \eqref{LCOP}). Therefore, these results also encourage us to further design efficient algorithms for \eqref{sdp}, with a focus on balancing the complexity of the primal and dual subproblems.

In this paper, we are particularly interested in \eqref{sdp} with the case that the dual subproblem is assumed to be easier than the primal one. Accordingly, we design a \textbf{t}riple-\textbf{B}regman primal-\textbf{d}ual \textbf{a}lgorithm (TBDA), which is able to exploit the unbalanced structure to achieve a flexible balance between the primal and dual problems. Specifically, we first follow the symmetric spirit used in \cite{HWY23} to update the dual variable twice, where the first one is only employed in the update of the primal variable. Here, we should emphasize that all subproblems involve Bregman proximal terms and each one has an individual Bregman kernel function, which makes our algorithm flexible to achieve a friendly balance between the primal and dual problems. Under appropriate parameter settings, our TBDA allows larger stepsizes than the PDHG method for the primal and dual subproblem. Furthermore, we incorporate an extrapolation step used in \eqref{pdhg} to enhance our TBDA, yielding improved numerical performance and a better theoretical convergence rate. In this way, the proposed TBDA enjoys a versatile algorithmic framework to produce some customized algorithms, and in particular to cover the iterative schemes of some classical first-order methods, including the (linearized and balanced) ALM and SPIDA. Under mild conditions, we show theoretically that our TBDA is globally convergence and exhibits an ergodic convergence rate of $O({1}/{N})$, where $N$ is the iteration counter. When $f$ or/and $g$ are further assumed to be strongly convex relative to their Bregman kernel functions, we introduce two improved versions of the original TBDA, and one of them  achieves a promising convergence rate of $O(1/\omega^{N})$, where $\omega>1$ is a given algorithmic constant. Finally, we apply our algorithm to quadratic optimization problems and RPCA. Some numerical results clearly show that our TBDA works better than the classical PDHG method and some other efficient variants.

This paper is divided into six sections. In Section \ref{Sec2}, we  summarize some basic concepts and definitions that will be used in the theoretical analysis. In Section \ref{Sec3}, we first introduce the algorithmic framework of the new TBDA. Then, we discuss the connection between the TBDA and several other first-order algorithms. Under appropriate conditions, we establish some convergence results of the TBDA. In Section \ref{Sec4}, we propose two improved versions of the TBDA under some stronger conditions. To support the idea of this paper, we report some computational results in Section \ref{Sec:numexp}. Finally, we complete this paper with some concluding remarks in Section \ref{Sec:conclusion}.

\section{Preliminaries}\label{Sec2}
This section summarizes some notations, definitions, and inequalities that will be used in subsequent analysis.

Let $\R^n$ be an $n$-dimensional Euclidean space. The superscript $\top$ denotes the transpose of vectors and matrices. For a given $x\in \mathbb{R}^n$, we denote $\| x\|_1:=\sum_{i=1}^{n}|x_i|$ as the $\ell_1$ norm of $x$, where $x_i$ is the $i$-th component of vector $x$. For a symmetric and positive definite matrix $G$ (denoted by $G\succ0$), we define $\langle x, y \rangle_{G} := \langle x, Gy \rangle$, where $x ,y \in \R^{n}$. Consequently, we further denote $\| x \|_{G} :=\sqrt{\langle x, Gx \rangle}$ as the $G$-norm of $x \in \R^{n}$. Given a matrix $A:=(a_{ij})_{m\times n}\in \mathbb{R}^{m\times n}$, we denote its nuclear norm by $\|A\|_{*}:= \sum_{i=1}^{\min\{m,n\}} \sigma_i(A)$, where $\sigma_i(A)$  is the $i$-th largest singular value of $A$.  Moreover, we use $\|A\|$ to represent the spectrum of $A$, which refers to the square root of the maximum eigenvalue of $A^\top A$.

\begin{definition}\label{def:lsc}
	Let $f:\R^n \to  \R \cup \{\infty\}$ be an extended real-valued function, and denote the domain of $f$ by
$\dom{f} := \left\{x\in \R^n\;|\;f(x) <\infty\right\}$. Then, we say that function $f$ is
\begin{enumerate}			
	\item[\rm (i)] proper if $f(x)  >-\infty$ for all $x \in \R^n$ and $\dom{f}\neq \emptyset$.
	\item[\rm (ii)] convex if $f\left(tx+(1-t)y\right)\leq t f(x) +(1-t)f(y)$ for any $x,y\in \dom{f}$ and $t\in (0,1)$.
	\item[\rm (iii)] $\varrho$-strongly convex with a given $\varrho>0$ if $\dom{f}$ is convex and the following inequality holds for any $x,y\in \dom{f}$ and $t\in(0,1)$:
	$$f(tx + (1-t)y)\leq t f(x) + (1-t)f(y) - \frac{\varrho}{2}t(1-t)\|x-y\|^2.$$
\end{enumerate}
\end{definition}

Given a proper closed convex function $f$, the subdifferential of $f$ at $x \in \dom{f}$ is defined by 
$$\partial f(x)= \left\{\; \xi\,|\, f(y) \geq f(x) + \left\langle y-x, \xi\right\rangle, \,\forall \, y \in \dom{f}  \;\right\}.$$
In what follows, we denote $\dom{\partial f}:=\{x\in\R^n\;|\;\partial f(x)\neq \emptyset\}$.

Let $\phi : \mathbb{R}^n \to \mathbb{R}$ be a strictly convex function, finite at $x$ and $y$, and differentiable at $y$. Then, the Bregman distance \cite{Bre66} between $x$ and $y$ associated with the kernel function $\phi$ is defined as
\begin{equation*}
\B_{\phi}(x,y) = \phi(x) - \phi(y) - \langle \nabla \phi(y),x-y \rangle,
\end{equation*}
where $\nabla \phi(y)$ represents the gradient of $\phi$ at $y \in \R^{n}$. For simplicity, we here refer the reader to \cite{BBT17,BSTV18} for some widely used Bregman kernel functions. It is easy to see that the Bregman distance covers the standard Euclidean distance as its special case when $\phi(\cdot)=\frac{1}{2}\|\cdot\|^2$. However, the Bregman distance does not always share the symmetry and the triangle inequality property with the Euclidean distance. Below, we recall two fundamental properties of the Bregman distance \cite{Beck17} that will be useful in the convergence analysis.

\begin{proposition}[{\cite[Proposition 2.3]{BBC03}}]
	Let $x\in\dom{\phi}$ and $a,b\in \intset{\dom{\partial\phi}}$. Then, we have
	\begin{itemize}
		\item[\rm (i)] $\B_{\phi}(a,b) + \B_{\phi}(b,a) = \langle  \nabla\phi(a)-\nabla\phi(b), a-b  \rangle$.
		\item[\rm (ii)] $\B_{\phi}(x,a) - \B_{\phi}(x,b) - \B_{\phi}(b,a) = \langle \nabla\phi(b)-\nabla\phi(a), x-b \rangle$.
	\end{itemize}
\end{proposition}
Note that $\intset{\mathbb{X}}$ denotes the interior of $\mathbb{X}$. In the case where $\mathbb{X}$ has no interior, we can instead use the relative interior of $\mathbb{X}$.
According to \cite{LFN18}, we define the strong convexity of $f$ relative to $\psi$.

\begin{definition}\label{def:RSC}
	Let $f:\R^n \to \R$, and let $\psi:\R^n \to \R$ be a differentiable convex function. We say that the function $f$ is $\varrho$-strongly convex relative to the function $\psi$ on $\dom{\psi}$, if for any $x,y \in \intset{\dom {\psi}}$ and $\nu \in \partial f(x)$, there exists a scalar $\varrho \geq 0$ for which
	\begin{equation*}
	f(y) \ge f(x) + \langle \nu , y-x \rangle + \varrho \B_{\psi}(y,x) .
	\end{equation*}
\end{definition}

Given a proper closed convex function $f : \mathbb{X} \to (-\infty, \infty]$, the proximal operator of $f$ (see \cite{Mor62,PB13}), denoted by $\prox_{\lambda f}(\cdot)$, is given by
\begin{equation*}
\prox_{\lambda f}(a) = \arg\min_{x \in \R^n } \left\{f(x) + \frac{1}{2\lambda }\|x-a\|^2 \right\},\quad \lambda>0,\;\; a\in\R^n.
\end{equation*}
With the definition of Bregman distance, we define the so-named Bregman proximity operator \cite{BCN06} associated with $\phi$ as follows:
\begin{equation*}
\Bprox_{\lambda f}^\phi (a) = \arg\min_{x \in \R^n } \left\{f(x) + \frac{1}{\lambda}\B_{\phi}(x,a) \right\},\quad \lambda>0,\;\; a\in\R^n,
\end{equation*}
which, by denoting $x^\star \equiv \Bprox_{\lambda f}^\phi (a)$, implies the following important property:
\begin{equation*}
	\frac{1}{\lambda}\langle  \nabla \phi(x^\star) - \nabla \phi(a) ,x - x^\star \rangle \ge f(x^\star) - f(x), \quad \forall x \in \mathbb{X}, 
\end{equation*}
or equivalently,
\begin{equation*}
\frac{1}{\lambda} \left(\B_{\phi}(x,a) - \B_{\phi}(x,x^\star) - \B_{\phi}(x^\star,a) \right) \ge f(x^\star) - f(x),\quad \forall x \in \mathbb{X}.
\end{equation*}

Below, we present the first-order optimality condition of \eqref{sdp}. The pair $(\widehat{x}, \widehat{y})\in \X  \times \Y$ is called a saddle point of \eqref{sdp} if it satisfies the following inequalities
\begin{equation*}
\L (\widehat{x}, y) \leq \L(\widehat{x}, \widehat{y}) \leq \L(x, \widehat{y}), \quad \forall\, x \in \X , \; \forall\, y \in \Y,
\end{equation*}
which means that
\begin{equation}\label{KKT}
\left\{
\begin{aligned}
&\mathcal{P}(x):= \L(x,\widehat{y}) - \L(\widehat{x},\widehat{y}) =f(x)-f(\widehat{x}) + \left\langle x-\widehat{x}, A^\top \widehat{y} \right\rangle  \geq 0, \;\; \forall x\in \mathbb{X}, \\
&\mathcal{D}(y):=\L(\widehat{x},\widehat{y})-\L(\widehat{x},y) =g(y)-g(\widehat{y}) + \left\langle y-\widehat{y}, -A\widehat{x} \right\rangle  \geq 0, \;\; \forall y\in \mathbb{Y}.
\end{aligned}\right.	
\end{equation}
In what follows, the primal–dual gap is defined by
\begin{equation}
\mathcal{G}(x,y): = \mathcal{P}(x) + \mathcal{D}(y) \ge 0 .  \nonumber
\end{equation}
It is not difficult to see that $\mathcal{P}(\cdot)$, $\mathcal{D}(\cdot)$ and $\mathcal{G}(\cdot,\cdot)$ in above are convex functions.

\section{The Triple-Bregman Balanced Primal-Dual Algorithm}\label{Sec3}

In this section, we first introduce the basic framework of the triple-Bregman balanced primal-dual algorithm (TBDA) for  \eqref{sdp}. Then, we discuss its connections with some existing first-order algorithms for \eqref{sdp} and \eqref{LCOP}. Finally, we show its global convergence under some mild conditions.

\subsection{Algorithmic framework: TBDA}
Recall the assumption that the saddle point problem \eqref{sdp} has an unbalanced structure in the sense that the primal subproblem is more difficult than the dual one. Therefore, we first follow the ways used in \cite{BR11,BR14,HWY23} to make a prediction $\tilde{y}^{k+1}$ on the dual variable $y$. Then, we update the primal one $x^{k+1}$ by absorbing the  $\tilde{y}^{k+1}$. Hereafter, we follow the extrapolation step in the PDHG method \eqref{pdhg} to generate $\bar{x}^{k+1}$ for numerical acceleration. Finally, we update the dual variable with the information of $\bar{x}^{k+1}$ and $y^k$. Note that the three primal and dual subproblems are equipped with Bregman proximal terms, which make our algorithm flexible to balance the computational cost of updating primal and dual variables (see Sections \ref{Sec5} and \ref{Sec:numexp}). The details of the algorithm are summarized in Algorithm \ref{alg1}.

\begin{algorithm}[!htbp]
	\caption{The Triple-Bregman Balanced Primal-Dual Algorithm for \eqref{sdp}.}\label{alg1}
	\begin{algorithmic}[1]
		\STATE Choose starting points $x^{0} \in \mathbb{X}$ and $y^{0} \in \mathbb{Y}$, and parameters $\sigma$, $\gamma$, $\mu$ and $\tau > 0$.
		\REPEAT
		\STATE Compute $\tilde{y}^{k+1}$, $x^{k+1}$, $\bar{x}^{k+1}$ and $y^{k+1}$, respectively, via
		\begin{align}
		&\tilde{y}^{k+1} = \arg\min_{y\in \mathbb{Y}}\left\{ g(y) - \langle A  x^{k}, y\rangle + \gamma   \B_{\phi}(y,y^{k}) \right\}, \label{a}  \\
		&x^{k+1} = \arg\min_{x\in \mathbb{X}}\left\{ f(x) + \langle Ax, \tilde{y}^{k+1} \rangle +  \mu   \B_{\psi}(x,x^{k}) \right\} \label{b},  \\
		&\bar{x}^{k+1} = x^{k+1} + \sigma (x^{k+1} - x^{k}), \label{c} \\
		&y^{k+1} = \arg\min_{y\in \mathbb{Y}}\left\{ g(y) - \langle A \bar{x}^{k+1}, y\rangle +  \tau   \B_{\varphi}(y,y^{k}) \right\}. \label{d}
		\end{align} 
		\UNTIL some stopping criterion is satisfied.
		\RETURN an approximate saddle point $(\hat{x},\hat{y})$. 
	\end{algorithmic}
\end{algorithm}

It has been documented in \cite{HXY22,HYY15} that the AHPD algorithm (i.e., taking $\sigma=0$ in \eqref{pdhg}) is not necessarily convergent in practice. Here, we consider a toy example constructed in \cite{HXY22} to illustrate the faster behavior of Algorithm \ref{alg1} than the PDHG method (with setting $\sigma=1$) and SPIDA \cite{HWY23}. Specifically, we consider the following linear programming problem: 
	\begin{equation}\label{Problem2}
	\min_{x_{1},x_{2}} \left\{\;  2x_{1} + x_{2} \;|\; x_{1} +x_{2} =1,\; x_{1} \ge 0 , \; x_{2} \ge 0\; \right\},
	\end{equation}
whose the dual form is $\max_{y} \left\{\;  y\;|\; y \le 1,\; y \le 2\; \right\}$.
	It is clear that the optimal solution of \eqref{Problem2} is  $(x_{1}^{\star},x_{2}^{\star})=(0,1)$ and $y^{\star}=1$ is the optimal solution of the dual problem. When applying Algorithm \ref{alg1} to \eqref{Problem2}, by taking parameters $\sigma=1$ and $\gamma = \mu = \tau $ and setting the Bregman kernel function as $\phi(w) = \psi(w) = \varphi(w)=\frac{1}{2}\|w\|^2$ for simplicity, the specific iterative schemes of Algorithm \ref{alg1} become
	\begin{equation}\label{tbda-lp}
		\begin{cases}
		\tilde{y}^{k+1} =y^{k} - \frac{1}{\gamma} \left( x^{k}_{1} + x^{k}_{2} -1   \right),  \\
		x^{k+1}_{1} = \max \left\{ \left(- \frac{2}{\gamma} + \frac{1}{\gamma} \tilde{y}^{k+1} +x_{1}^{k}  \right),0  \right\},   \\
		x^{k+1}_{2} = \max \left\{ \left(- \frac{1}{\gamma} + \frac{1}{\gamma} \tilde{y}^{k+1} +x_{2}^{k}  \right),0  \right\},    \\
		\bar{x}^{k+1}_{1} = 2 x^{k+1}_{1} - x^{k}_{1},  \\
		\bar{x}^{k+1}_{2} = 2 x^{k+1}_{2} - x^{k}_{2}, \\
		y^{k+1} = y^{k} - \frac{1}{\gamma} \left( \bar{x}^{k+1}_{1} + \bar{x}^{k+1}_{2} - 1 \right). 
		\end{cases}
	\end{equation}
	We can see from \eqref{tbda-lp} that the update of the dual variable (i.e., $\tilde{y}^{k+1}$ and $y^{k+1}$) is simpler than the primal variable $x^{k+1}$ without projection operations (i.e., $\max\{\cdot,0\}$).
	Now, we consider two different groups of proximal parameters to conduct the performance of PDHG, SPIDA, and TBDA. All the algorithms start with $(x^0_1, x^0_2, y^0) = (0, 0, 0)$. In Figure \ref{figure1}, the first row plots the convergence trajectories of the three algorithms by setting $\mu=\gamma= \frac{2\sqrt{6}}{3}$ for PDHG and SPIDA and $\mu=\gamma=\tau = \frac{2\sqrt{6}}{3}$ for TBDA. The plots in the second row show the convergence trajectories of the three algorithms by setting $\mu=\gamma= \frac{10\sqrt{6}}{3}$ for PDHG and SPIDA and $\mu=\gamma=\tau = \frac{10\sqrt{6}}{3}$ for TBDA. It is clearly demonstrated from Figure \ref{figure1} that our TBDA runs faster than PDHG and SPIDA for the toy example \eqref{Problem2}, which, to some extent, confirms that our proposed TBDA exhibits superior numerical performance compared to PDHG and SPIDA.

\begin{figure}[!htbp]
	\subfigure[PDHG $(\gamma = \mu = \frac{2\sqrt{6}}{3})$]{
		\begin{minipage}[t]{0.31\linewidth}
			\centering
			\includegraphics[width=1\textwidth]{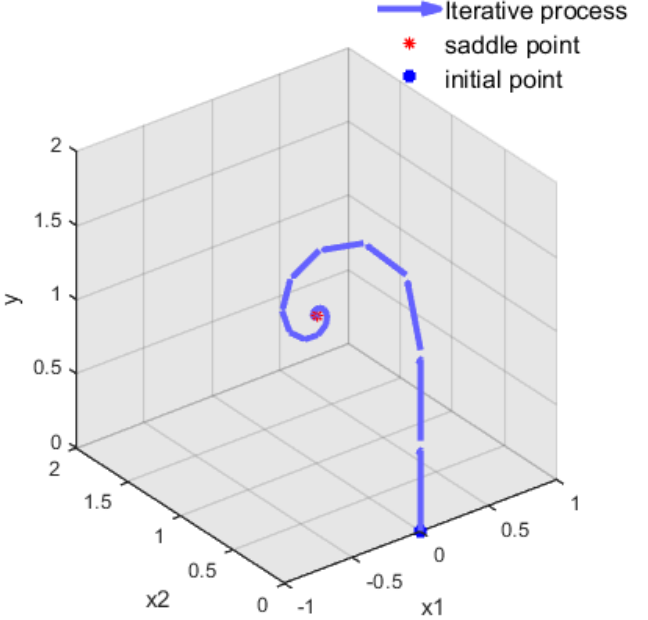}
		\end{minipage}%
	}
	\subfigure[SPIDA $(\gamma = \mu = \frac{2\sqrt{6}}{3})$]{
		\begin{minipage}[t]{0.31\linewidth}
			\centering
			\includegraphics[width=1\textwidth]{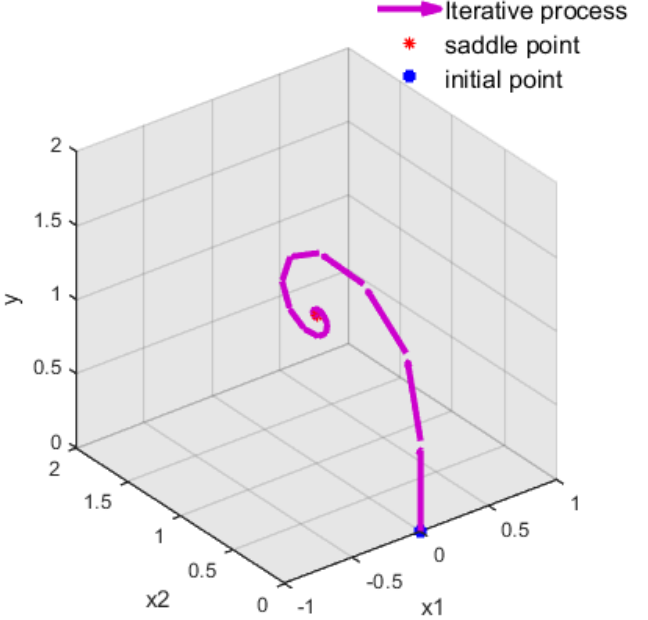}
		\end{minipage}
	}
	\subfigure[TBDA $(\gamma = \mu = \tau = \frac{2\sqrt{6}}{3})$]{
		\begin{minipage}[t]{0.31\linewidth}
			\centering
			\includegraphics[width=1\textwidth]{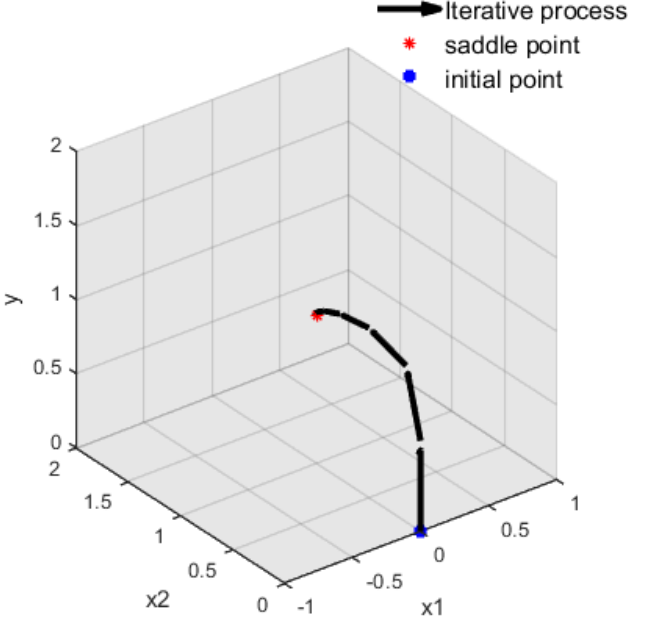}
		\end{minipage}
	}
	
	\subfigure[PDHG $(\gamma = \mu = \frac{10\sqrt{6}}{3})$]{
		\begin{minipage}[t]{0.31\linewidth}
			\centering
			\includegraphics[width=1\textwidth]{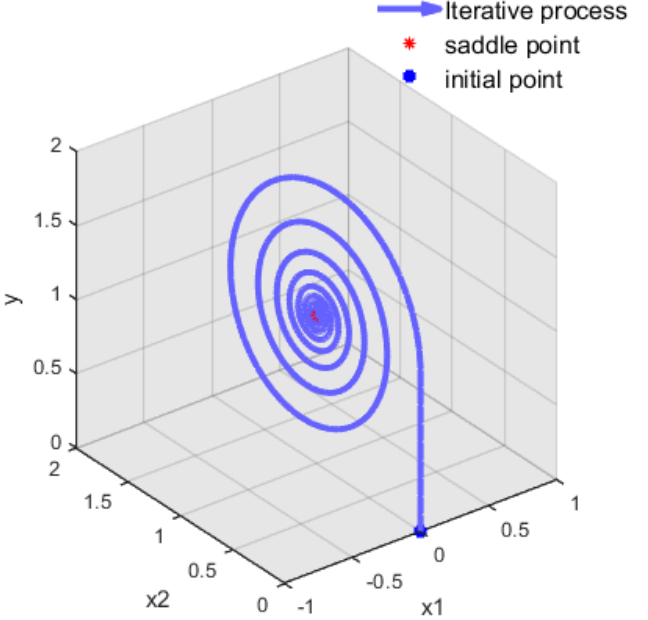}
		\end{minipage}%
	}
	\subfigure[SPIDA $(\gamma = \mu = \frac{10\sqrt{6}}{3})$]{
		\begin{minipage}[t]{0.31\linewidth}
			\centering
			\includegraphics[width=1\textwidth]{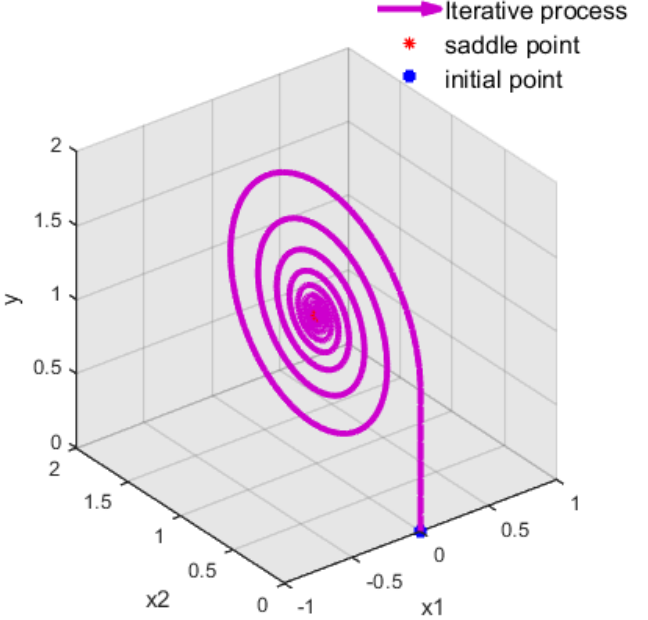}
		\end{minipage}
	}
	\subfigure[TBDA $(\gamma = \mu = \tau = \frac{10\sqrt{6}}{3})$]{
		\begin{minipage}[t]{0.31\linewidth} 
			\centering
			\includegraphics[width=1\textwidth]{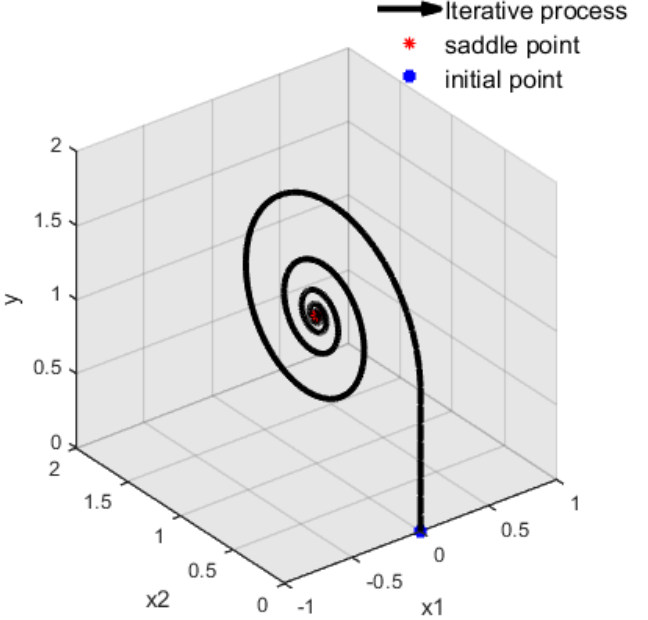}
		\end{minipage}
	}
	\caption{Illustration of the convergence behaviors of PDHG, SPIDA and TBDA with different proximal parameters for solving \eqref{Problem2}.}
	\label{figure1}
\end{figure}

\subsection{Understanding some first-order methods by TBDA}\label{Sec5}

In this subsection, by choosing the Bregman kernel functions and the extrapolation parameter, we demonstrate the versatility of TBDA to produce many variants. Particularly, we focus on applying TBDA to the linearly constrained convex optimization problem \eqref{LCOP} for understanding the iterative schemes of some first-order methods in the framework of our TBDA. 

Firstly, it is clear that our TBDA immediately reduces to the SPIDA \cite{HWY23} by removing the extrapolation step \eqref{c} and taking the same Bregman kernel function in \eqref{a} and \eqref{d}. Below, we discuss its connections with some other first-order algorithms for the linearly constrained convex optimization problem \eqref{LCOP}. We now formulate \eqref{LCOP} as a saddle point problem:
\begin{equation}\label{sdp-lcp}
\min_{x\in \mathbb{X}}\max_{y \in \R^m} \L(x,y) := f(x) + \langle Ax,y \rangle -\langle b,y \rangle,
\end{equation}
which is a special case of \eqref{sdp} with specifying $g(y)=\langle b,y\rangle$. Consequently, applying TBDA to \eqref{sdp-lcp} yields 
\begin{equation}\label{tbda-lcp}
\left\{
\begin{aligned}
&\tilde{y}^{k+1} = \arg\min_{y\in \R^{n}}\left\{ \langle b,y \rangle - \langle A  x^{k}, y\rangle + \gamma   \B_{\phi}(y,y^{k}) \right\},  \\
&x^{k+1} = \arg\min_{x\in \mathbb{X}}\left\{ f(x) + \langle Ax, \tilde{y}^{k+1} \rangle +  \mu   \B_{\psi}(x,x^{k}) \right\},  \\
&\bar{x}^{k+1} = x^{k+1} + \sigma (x^{k+1} - x^{k}) , \\
&y^{k+1} = \arg\min_{y\in \R^{m}}\left\{ \langle b,y \rangle - \langle A \bar{x}^{k+1}, y\rangle +  \tau   \B_{\varphi}(y,y^{k}) \right\}.
\end{aligned}\right.
\end{equation} 
 
Hereafter, we illustrate the versatility of \eqref{tbda-lcp} to produce first-order algorithms for \eqref{LCOP}.

\subsubsection{Understanding the augmented Lagrangian method (ALM)}\label{Sec:ALM} 

With the iterative scheme \eqref{tbda-lcp}, we take the Bregman kernel functions as $\phi(y)= \varphi(y) = \frac{1}{2} \| y \|^{2} $ and $\psi(x)=  \frac{1}{2} \| x \|_{A^\top A}^{2}$. Then, the Bregman proximal terms in \eqref{tbda-lcp} immediately are specified as:
$$\B_{\phi}(y,y^{k}) = \B_{\varphi}(y,y^{k}) = \frac{1}{2} \| y - y^{k} \|^{2} \quad\text{and}\quad \B_{\psi}(x,x^{k}) = \frac{1}{2} \| x - x^{k} \|_{A^\top A}^{2} .$$
Consequently, \eqref{tbda-lcp} becomes
\begin{equation}\label{tbda-lcp-2}
\left\{
\begin{aligned}
&\tilde{y}^{k+1} = \arg\min_{y\in \R^{n}}\left\{ \langle b,y \rangle - \langle A  x^{k}, y\rangle + \frac{\gamma}{2} \| y - y^{k} \|^{2} \right\},  \\
&x^{k+1} = \arg\min_{x\in \mathbb{X}}\left\{ f(x) + \langle Ax, \tilde{y}^{k+1} \rangle +  \frac{\mu}{2} \| x - x^{k} \|_{A^\top A}^{2} \right\},  \\
&\bar{x}^{k+1} = x^{k+1} + \sigma (x^{k+1} - x^{k}),  \\
&y^{k+1} = \arg\min_{y\in \R^{m}}\left\{ \langle b,y \rangle - \langle A \bar{x}^{k+1}, y\rangle +  \frac{\tau}{2}  \| y - y^{k} \|^{2} \right\}.
\end{aligned}\right.
\end{equation}  
Note that both $y$-subproblems are unconstrained optimization problems. Therefore, by invoking the first-order optimality conditions, \eqref{tbda-lcp-2} can be further simplified as
\begin{equation}\label{tbda-lcp-3}
\left\{
\begin{aligned}
&\tilde{y}^{k+1} = y^{k} + \frac{1}{\gamma}  \left(A x^{k} - b\right),\\
&x^{k+1} = \arg\min_{x\in \mathbb{X}}\left\{ f(x) +  \frac{\mu}{2} \| Ax - Ax^{k} + \frac{1}{\mu}  \tilde{y}^{k+1} \|^{2} \right\},  \\
&\bar{x}^{k+1} = x^{k+1} + \sigma (x^{k+1} - x^{k}) , \\
&y^{k+1} = y^{k} + \frac{1}{\tau}  \left(A \bar{x}^{k+1} - b \right).
\end{aligned}\right.
\end{equation}  
Observing that $\tilde{y}^{k+1}$ essentially serves as an intermediate iterate, we plug it into the update of $x^{k+1}$ with setting $\mu=\frac{1}{\gamma}$. Then,  \eqref{tbda-lcp-3} reads as 
\begin{equation}\label{tbda-lcp-4}
\left\{
\begin{aligned}
&x^{k+1} = \arg\min_{x\in \mathbb{X}}\left\{ f(x) +  \frac{1}{2\gamma} \left\| Ax - b + \gamma y^k \right\|^{2} \right\},  \\
&\bar{x}^{k+1} = x^{k+1} + \sigma (x^{k+1} - x^{k}) , \\
&y^{k+1} = y^{k} + \frac{1}{\tau}  (A \bar{x}^{k+1} - b ).
\end{aligned}\right.
\end{equation}  
As a result, when taking $\gamma = \tau$ and $\sigma =0$ in \eqref{tbda-lcp-4}, we obtain the iterative scheme of the standard ALM. Additionally,  \eqref{tbda-lcp-4} produces some other variants by setting $\gamma \neq \tau$ and $\sigma \neq 0$. 


\subsubsection{Understanding the linearized ALM}\label{Sec:LALM}
With the discussions in Section \ref{Sec:ALM}, we only change the Bregman kernel function $\psi$ in the $x$-subproblem as $\psi(x) = \frac{1}{2} \|x \|^{2} $. Then, we specify \eqref{tbda-lcp} as
\begin{equation}\label{tbda-lalm}
\left\{
\begin{aligned}
&\tilde{y}^{k+1} = y^{k} +  \frac{1}{\gamma} (A x^{k} - b ), \\
&x^{k+1} = \arg\min_{x\in \mathbb{X}}\left\{ f(x) +  \frac{\mu}{2} \| x - x^{k} + \frac{1}{\mu} A^{\top} \tilde{y}^{k+1} \|^{2} \right\},  \\
&\bar{x}^{k+1} = x^{k+1} + \sigma (x^{k+1} - x^{k}),  \\
&y^{k+1} = y^{k} + \frac{1}{\tau} (A \bar{x}^{k+1} - b ).
\end{aligned}\right.
\end{equation}  
As a consequence, when taking $\gamma= \tau$ and $\sigma = 0$, the iterative scheme \eqref{tbda-lalm} recovers the linearized ALM. Of course, we can also take $\gamma \neq \tau$ and $\sigma \neq 0$ to produce some other variants of \eqref{tbda-lalm}.

\subsubsection{Understanding the balanced ALM}\label{Sec:BALM}

Unlike the choices discussed in Sections \ref{Sec:ALM} and \ref{Sec:LALM}, we take two different Bregman kernel functions for the two $y$-subproblems. Specifically, we set $\phi(y)= \frac{1}{2} \| y \|^{2} $, $\varphi(y) = \frac{1}{2} \| y \|^{2}_{(AA^\top +\kappa I)}$, and $\psi(x)=  \frac{1}{2} \| x \|^2$, where $\kappa>0$ and $I$ stands for an identity matrix which always automatically
matches the size for operations. Then, we have
\begin{equation}\label{tbda-balm}
\left\{
\begin{aligned}
&\tilde{y}^{k+1} = y^{k} + \frac{1}{\gamma} (A x^{k} - b ), \\
&x^{k+1} = \arg\min_{x\in \mathbb{X}}\left\{ f(x) +  \frac{\mu}{2} \| x - x^{k} + \frac{1}{\mu} A^{\top} \tilde{y}^{k+1} \|^{2} \right\},  \\
&\bar{x}^{k+1} = x^{k+1} + \sigma (x^{k+1} - x^{k}),  \\
&y^{k+1} = \arg\min_{y\in \R^{m}}\left\{ \langle b,y \rangle - \langle A \bar{x}^{k+1}, y\rangle +  \frac{\tau}{2}  \left\| y - y^{k} \right\|^{2}_{(AA^\top +\kappa I)} \right\},
\end{aligned}\right.
\end{equation}  
where the last $y$-subproblem amounts to
$$y^{k+1} = y^k + \frac{1}{\tau} (AA^\top +\kappa I)^{-1} (A\bar{x}^{k+1}-b). $$
Consequently, by plugging $\tilde{y}^{k+1}$ into the update of $x^{k+1}$ and taking $\mu=\frac{1}{\gamma}$ and $\sigma=\tau=1$, we immediately obtain the balanced ALM introduced in \cite{HY21}. Moreover, when taking $\phi(y)=\varphi(y) = \frac{1}{2} \| y \|^{2}_{(AA^\top +\kappa I)}$ and $\sigma=0$, we can recover the doubly balanced ALM mentioned in \cite{HWY23}. Also, we can produce some extrapolated variants with $\sigma\neq 0$.

\subsubsection{Understanding ALM-based multi-block splitting methods}\label{Sec:MALM}
In recent years, the multi-block structure frequently appears in those models arising from imaging sciences and statistical learning. Here, we consider a multi-block form of the saddle point problem \eqref{sdp}, i.e.,
\begin{equation}\label{multi-sdp1}
\min_{x_{i}\in\mathbb{X}_{i}}\max_{y \in \R^m} \left\{ \sum_{i=1}^{p}f_{i}(x_{i}) + \left\langle  \sum_{i=1}^{p}A_{i}x_{i} ,y \right\rangle -g(y) \right\},
\end{equation}
where each $f_{i} : \mathbb{X}_{i} \rightarrow (-\infty,+\infty]$ is a proper closed convex function and $A_i\in \R^{m\times n_i}$ is a given matrix. In particular, when $g(y)=\langle b,y\rangle$, the problem \eqref{multi-sdp1} is equivalent to the multi-block linearly constrained convex optimization:
\begin{equation}\label{MLCP}
	\min_{x_i\in\mathbb{X}_i}\left\{\;\sum_{i=1}^{p}f_{i}(x_{i})\;\Big{|}\; \sum_{i=1}^{p}A_{i}x_{i}=b\;\right\}.
\end{equation}
By denoting $\bx:=(x_1^\top,x_2^\top,\ldots,x_p^\top)^\top$, $\bbf(\bx):=\sum_{i=1}^{p}f_{i}(x_{i}) $, $\bA:=[A_1,A_2,\ldots,A_p]$, and $\mathbb{X}:=\mathbb{X}_1\times \mathbb{X}_2\times\ldots\times \mathbb{X}_p$, we can rewrite  \eqref{multi-sdp1} as the following compact form:
\begin{equation}\label{multi-sdp2}
\min_{\bx \in \mathbb{X}}\max_{y \in \R^m} \L(\bx,y) := \bbf(\bx) + \left\langle  \bA\bx,y \right\rangle -g(y) .
\end{equation}
Consequently, applying Algorithm \ref{alg1} to \eqref{multi-sdp2} and using the separability of objective function, we have
\begin{equation}\label{PD-Full-Jacobian}
\left\{ \begin{aligned}
\tilde{y}^{k+1} &= \arg\min_{y\in \R^m}\left\{ g(y) - \left\langle \sum_{i=1}^{p}A_{i}x_{i}^k, y\right\rangle + \gamma  \B_\phi(y,y^k) \right\},  \\
x_i^{k+1} &= \arg\min_{x_i\in \mathbb{X}_i}\left\{ f(x_{i}) + \left\langle A_{i}x_{i}, \tilde{y}^{k+1} \right\rangle + \mu \B_{\psi_{i}}(x_{i},x_{i}^k) \right\},  \; i=1,\ldots,p,  \\
\bar{x}_{i}^{k+1} & = x_i^{k+1} + \sigma (x_i^{k+1} - x_i^{k}), \; i=1,\ldots,p,  \\
y^{k+1} &= \arg\min_{y\in \R^m}\left\{ g(y) - \left\langle \sum_{i=1}^{p}A_{i}\bar{x}_{i}^{k+1}, y\right\rangle + \tau \B_\varphi(y,y^k) \right\}.
\end{aligned}\right.
\end{equation}
It is noteworthy that all $x_i$-subproblems are separable. Then, we can implement \eqref{PD-Full-Jacobian} in a parallel way. When setting $\phi(y)=\varphi(y)$ and $\sigma=0$, it has been shown in \cite{HWY23} that \eqref{PD-Full-Jacobian} recovers the iterative schemes of many ALM-based multi-block splitting methods, e.g., \cite{HHX14,HYZ14,HHY15,HXY16,WDH15}. More generally, we can take $\phi(y)\neq \varphi(y)$ and $\sigma \neq 0$ to generate more implementable variants for \eqref{MLCP}. In Section \ref{Sec:numexp}, we will demonstrate the extrapolated step $\bar{x}_i^{k+1}$ with $\sigma\neq 0$ is able to improve the performance of \eqref{PD-Full-Jacobian}.

\subsection{Convergence analysis}
In this subsection, we demonstrate that the sequence generated by Algorithm \ref{alg1} converges to a solution of \eqref{sdp} under relatively mild conditions.

\begin{lemma}\label{le1}
	Let $\{(\hat{x},\hat{y})\}$ be a saddle point of \eqref{sdp}, and let $\{(\tilde{y}^{k},x^{k}, y^{k})\}$ be a sequence generated by Algorithm \ref{alg1}. Then, we have 
	\begin{align}\label{eq12}
	&\tau \B_{\varphi}(\hat{y},y^{k+1}) + \mu\B_{\psi}(\hat{x},x^{k+1})+  \mathcal{D}(\tilde{y}^{k+1}) + (\sigma+1)\mathcal{P}(x^{k+1})    \\
	&\le  \tau \B_{\varphi}(\hat{y},y^{k}) + \mu\B_{\psi}(\hat{x},x^{k}) + \sigma \mathcal{P}(x^{k}) - \left(\tau\B_{\varphi}(y^{k+1},y^{k}) - \gamma \B_{\phi}(y^{k+1},y^{k})\right)   \nonumber   \\  
    &\quad - \left( \mu(1+\sigma) \B_{\psi}(x^{k+1},x^{k}) + \mu\sigma\B_{\psi}(x^{k},x^{k+1})     \right. \nonumber \\
    &\qquad\;\;  \left. + \gamma  \B_{\phi}(y^{k+1},\tilde{y}^{k+1}) - (1+\sigma) \left\langle  A (x^{k+1} - x^{k}) ,y^{k+1} - \tilde{y}^{k+1} \right\rangle \right). \nonumber 
	\end{align}
\end{lemma}

\begin{proof}
First, by invoking the first-order optimality conditions of  \eqref{a}, \eqref{b} and \eqref{d}, it holds, for all $y\in\mathbb{Y}$ and $x\in\mathbb{X}$, that
\begin{align}
	& \left\langle \gamma \left(\nabla \phi(\tilde{y}^{k+1}) - \nabla \phi(y^{k})\right) -  A x^{k} ,y-\tilde{y}^{k+1} \right\rangle \ge g(\tilde{y}^{k+1}) - g(y),   \label{eq1} \\
	& 	\left\langle \tau \left(\nabla \varphi(y^{k+1}) - \nabla \varphi(y^{k})\right) -  A \bar{x}^{k+1} ,y-y^{k+1} \right\rangle \ge g(y^{k+1}) - g(y),  \label{eq2} \\
	& \left\langle \mu  \left(\nabla \psi(x^{k+1}) - \nabla \psi(x^{k})\right) +  A^\top \tilde{y}^{k+1} ,x- x^{k+1} \right\rangle \ge f(x^{k+1}) - f(x).\label{eq3}
\end{align}
	Due to the arbitrariness of $y$ in  \eqref{eq1} and  \eqref{eq2}, we let $y:=y^{k+1} $ in \eqref{eq1} and $y:=\hat{y}$ in \eqref{eq2}, respectively. Then, we arrive at
	\begin{equation}\label{eq4}
	\left\langle \gamma \left(\nabla \phi(\tilde{y}^{k+1}) - \nabla \phi(y^{k})\right) -  A x^{k} ,y^{k+1} - \tilde{y}^{k+1} \right\rangle \ge g(\tilde{y}^{k+1}) - g(y^{k+1})
	\end{equation}
	and
	\begin{equation}\label{eq5}
	\left\langle \tau \left(\nabla \varphi(y^{k+1}) - \nabla \varphi(y^{k})\right) -  A \bar{x}^{k+1} ,\hat{y}-y^{k+1} \right\rangle \ge g(y^{k+1}) - g(\hat{y}).
	\end{equation}
	By adding \eqref{eq4} and \eqref{eq5}, we have
	\begin{align}\label{eq6}
	&\gamma \left\langle \nabla \phi(\tilde{y}^{k+1}) - \nabla \phi(y^{k}),y^{k+1}-\tilde{y}^{k+1}  \right\rangle + \tau \left\langle \nabla \varphi(y^{k+1}) - \nabla \varphi(y^{k}), \hat{y}-y^{k+1} \right\rangle  \\
	&\hskip 1cm - \left\langle  A \bar{x}^{k+1} ,\hat{y} - \tilde{y}^{k+1} \right\rangle + \left\langle   A (x^{k} - \bar{x}^{k+1}) ,\tilde{y}^{k+1} -y^{k+1} \right\rangle \nonumber \\
	&	\ge  g(\tilde{y}^{k+1}) - g(\hat{y}). \nonumber
	\end{align}
	Similarly, by the arbitrariness of $x$ in \eqref{eq3}, 
	we take $x:=\hat{x}$ and let $x:= x^{k}$ with multiplying both sides by $\sigma$. Then, we have
	\begin{equation}\label{eq13}
	\left\langle \mu ( \nabla \psi(x^{k+1}) - \nabla \psi(x^{k})) + A^\top \tilde{y}^{k+1}, \hat{x}- x^{k+1} \right\rangle \ge f(x^{k+1}) - f(\hat{x})
	\end{equation}
	and
	\begin{align}\label{eq14}
	&\left\langle \mu  \left(\nabla \psi(x^{k+1}) - \nabla \psi(x^{k})\right) +  A^\top \tilde{y}^{k+1} ,\sigma( x^{k}- x^{k+1}) \right\rangle  \\
	&\hskip 6.5cm \ge  \sigma  \left(f(x^{k+1}) - f(x^{k})\right).\nn
	\end{align}
	Combining \eqref{eq6}, \eqref{eq13} and \eqref{eq14}, we have
	\begin{align}\label{eq15}
	&\gamma \langle \nabla \phi(\tilde{y}^{k+1}) - \nabla \phi(y^{k}),y^{k+1}-\tilde{y}^{k+1}  \rangle  + \tau \langle \nabla \varphi(y^{k+1}) - \nabla \varphi(y^{k}) ,\hat{y}-y^{k+1} \rangle  \\
	&\quad  + \mu \langle   \nabla \psi(x^{k+1}) - \nabla \psi(x^{k}) , \hat{x}- x^{k+1} \rangle + \langle  A (x^{k} - \bar{x}^{k+1}) ,\tilde{y}^{k+1} -y^{k+1} \rangle  \nonumber \\
	& \quad+ \mu \langle   \nabla \psi(x^{k+1}) - \nabla \psi(x^{k}) , \sigma( x^{k}- x^{k+1}) \rangle - \langle  A \bar{x}^{k+1} ,\hat{y} - \tilde{y}^{k+1} \rangle  \nonumber \\
	&\quad + \langle  A^\top \tilde{y}^{k+1}, \hat{x} - \bar{x}^{k+1} \rangle  \nonumber \\ 	
	&\ge   g(\tilde{y}^{k+1}) - g(\hat{y}) + f(x^{k+1}) - f(\hat{x}) + \sigma \left(f(x^{k+1}) - f(x^{k})\right).\nonumber
	\end{align}
	Note that
	\begin{equation*}
	\langle  A (\hat{x} - \bar{x}^{k+1}) ,\tilde{y}^{k+1} - \hat{y} \rangle  - \langle    \hat{x} - \bar{x}^{k+1} ,A^{\top}(\tilde{y}^{k+1} - \hat{y}) \rangle  = 0,
	\end{equation*}
	which, together with \eqref{eq15}, leads to
	\begin{align}\label{eq8}
	&\gamma \langle \nabla \phi(\tilde{y}^{k+1}) - \nabla \phi(y^{k}),y^{k+1}-\tilde{y}^{k+1}  \rangle  + \tau \langle \nabla \varphi(y^{k+1}) - \nabla \varphi(y^{k}) ,\hat{y}-y^{k+1} \rangle  \\
	& \quad + \mu  \langle  \nabla \psi(x^{k+1}) - \nabla \psi(x^{k}) , \hat{x}- x^{k+1} \rangle  + \langle   A (x^{k} - \bar{x}^{k+1}) ,\tilde{y}^{k+1} -y^{k+1} \rangle \nonumber \\
	& \quad + \mu \langle  \nabla \psi(x^{k+1}) - \nabla \psi(x^{k}) , \sigma( x^{k}- x^{k+1}) \rangle   \nonumber \\
	 &\ge   f(x^{k+1}) - f(\hat{x}) + \sigma \left(f(x^{k+1}) - f(x^{k})\right) +  \langle A^{\top} \hat{y} , \bar{x}^{k+1} - \hat{x} \rangle  \nonumber \\
	&\quad   + g(\tilde{y}^{k+1}) - g(\hat{y}) -  \langle A \hat{x}  ,\tilde{y}^{k+1} - \hat{y} \rangle.\nonumber 
	\end{align}
	By invoking the definitions of $\mathcal{P}(\cdot)$ and $\mathcal{D}(\cdot)$ in \eqref{KKT}, the right-hand side of \eqref{eq8} can be simplified as 
		\begin{align*}
	&  f(x^{k+1}) - f(\hat{x}) + \sigma \left(f(x^{k+1}) - f(x^{k})\right) + \langle A^{\top} \hat{y} , \bar{x}^{k+1} - \hat{x} \rangle  \\ 
	& \quad +g(\tilde{y}^{k+1}) - g(\hat{y}) - \langle   A \hat{x}  ,\tilde{y}^{k+1} - \hat{y} \rangle \nonumber \\
	&=   \sigma \left[ \left(f(x^{k+1}) - f(\hat{x}) + \langle  A^{\top} \hat{y} , x^{k+1} - \hat{x} \rangle\right)  - \left(f(x^{k}) - f(\hat{x}) + \langle A^{\top} \hat{y} ,  x^{k}  - \hat{x} \rangle \right) \right] \nonumber \\
	&\qquad+ f(x^{k+1}) - f(\hat{x}) +\langle  A^{\top} \hat{y} , x^{k+1} - \hat{x} \rangle +g(\tilde{y}^{k+1}) - g(\hat{y}) - \langle   A \hat{x}  ,\tilde{y}^{k+1} - \hat{y} \rangle \nonumber \\
	&= \mathcal{P}(x^{k+1}) + \sigma(\mathcal{P}(x^{k+1}) - \mathcal{P}(x^{k})) +\mathcal{D}(\tilde{y}^{k+1}), \nonumber
	\end{align*} 
	which, together with the three-point inequality of Bregman distance and \eqref{eq8}, implies that 
	\begin{align*}
	& \mathcal{D}(\tilde{y}^{k+1}) + \mathcal{P}(x^{k+1}) + \sigma(\mathcal{P}(x^{k+1}) - \mathcal{P}(x^{k})) \\
	&\le \gamma \langle \nabla \phi(\tilde{y}^{k+1}) - \nabla \phi(y^{k}),y^{k+1}-\tilde{y}^{k+1}  \rangle  + \tau \langle \nabla \varphi(y^{k+1}) - \nabla \varphi(y^{k}) ,\hat{y}-y^{k+1} \rangle \nonumber \\
	& \quad + \mu\langle  \nabla \psi(x^{k+1}) - \nabla \psi(x^{k}), \hat{x}- x^{k+1} \rangle + \langle   A (x^{k} - \bar{x}^{k+1}) ,\tilde{y}^{k+1} -y^{k+1} \rangle  \nonumber \\
	& \quad + \mu\langle  \nabla \psi(x^{k+1}) - \nabla \psi(x^{k}) , \sigma( x^{k}- x^{k+1}) \rangle      \nonumber \\
	&= \gamma (\B_{\phi}(y^{k+1},y^{k}) - \B_{\phi}(y^{k+1},\tilde{y}^{k+1}) - \B_{\phi}(\tilde{y}^{k+1},y^{k}))   \nonumber \\
	&\quad  + \tau \left(\B_{\varphi}(\hat{y},y^{k}) - \B_{\varphi}(\hat{y},y^{k+1}) - \B_{\varphi}(y^{k+1},y^{k})\right) \nonumber \\
	&\quad + \mu(\B_{\psi}(\hat{x},x^{k}) - \B_{\psi}(\hat{x},x^{k+1}) - \B_{\psi}(x^{k+1},x^{k}))   \nonumber \\
	&\quad + \langle  A (x^{k} - \bar{x}^{k+1}) ,\tilde{y}^{k+1} - y^{k+1}  \rangle - \mu\sigma \left(\B_{\psi}(x^{k},x^{k+1}) + \B_{\psi}(x^{k+1},x^{k}) \right) \nonumber \\ 
	&=\tau(\B_{\varphi}(\hat{y},y^{k}) - \B_{\varphi}(\hat{y},y^{k+1})) - \gamma \B_{\phi}(y^{k+1},\tilde{y}^{k+1}) - \gamma \B_{\phi}(\tilde{y}^{k+1},y^{k}) \nn \\
	&\quad + \mu \left(\B_{\psi}(\hat{x},x^{k}) - \B_{\psi}(\hat{x},x^{k+1})\right) - \mu(1+\sigma)\B_{\psi}(x^{k+1},x^{k}) -  \sigma \mu\B_{\psi}(x^{k},x^{k+1}) \nonumber \\
	&\quad  + (1+\sigma)\langle  A (x^{k+1} - x^{k}) ,y^{k+1}- \tilde{y}^{k+1} \rangle  \nonumber \\
	&\quad - \left(\tau\B_{\varphi}(y^{k+1},y^{k}) - \gamma\B_{\phi}(y^{k+1},y^{k})\right), \nonumber 
	\end{align*}
	which, by further rearranging terms, concludes that
	\begin{align*}
	&\tau \B_{\varphi}(\hat{y},y^{k+1}) + \mu\B_{\psi}(\hat{x},x^{k+1}) +  \mathcal{D}(\tilde{y}^{k+1}) + (\sigma+1)\mathcal{P}(x^{k+1})    \\
	&\le  \tau \B_{\varphi}(\hat{y},y^{k}) + \mu\B_{\psi}(\hat{x},x^{k}) + \sigma \mathcal{P}(x^{k})  - \left(\tau\B_{\varphi}(y^{k+1},y^{k}) - \gamma \B_{\phi}(y^{k+1},y^{k})\right)    \\  
	&\quad  - \left( \mu(1+\sigma) \B_{\psi}(x^{k+1},x^{k}) + \mu\sigma\B_{\psi}(x^{k},x^{k+1})  + \gamma  \B_{\phi}(y^{k+1},\tilde{y}^{k+1})   \right.  \\
	&\qquad\quad \left. + \gamma \B_{\phi}(\tilde{y}^{k+1},y^{k})   - (1+\sigma)\langle  A (x^{k+1} - x^{k}) ,y^{k+1} - \tilde{y}^{k+1} \rangle \right) .
	\end{align*}
	The proof is complete.
\end{proof}

\begin{theorem}\label{th2}
	Let $\{(\hat{x},\hat{y})\}$ be a saddle point of \eqref{sdp} and $\{(\tilde{y}^{k},x^{k}, y^{k})\}$ be a sequence generated by Algorithm \ref{alg1}. For given $\gamma$, $\mu$, $\tau> 0$, if there exist positive constants $c_{1}$, $c_{2}$, $c_{3} > 0$ such that the following inequality  
	\begin{align}\label{eqs}
	& \mu(1+\sigma) \B_{\psi}(x^{k+1},x^{k})+\mu\sigma\B_{\psi}(x^{k},x^{k+1})  + \gamma  \B_{\phi}(y^{k+1},\tilde{y}^{k+1})     \nn  \\
	&\quad + \gamma \B_{\phi}(\tilde{y}^{k+1},y^{k}) -(1+\sigma)\langle  A (x^{k+1} - x^{k}) ,y^{k+1} - \tilde{y}^{k+1} \rangle \nn \\ 
	&\quad  + \tau\B_{\varphi}(y^{k+1},y^{k}) - \gamma \B_{\phi}(y^{k+1},y^{k})\nn \\
	&\ge  c_{1} \B_{\psi}(x^{k+1},x^{k}) +  c_{2}\B_{\phi}(y^{k+1},\tilde{y}^{k+1}) + c_{3} \B_{\phi}(\tilde{y}^{k+1},y^{k}).
	\end{align}
	Then the sequence $\{(x^{k}, y^{k})\}$ is bounded in $\mathbb{X}  \times \mathbb{Y}$, and all its cluster points are solutions of the saddle point problem \eqref{sdp}.
\end{theorem}

\begin{proof}
	It first follows from \eqref{eqs} and Lemma \ref{le1} that
	\begin{align*}
	&\tau \B_{\varphi}(\hat{y},y^{k+1}) + \mu\B_{\psi}(\hat{x},x^{k+1}) + \mathcal{D}(\tilde{y}^{k+1}) + (\sigma+1)\mathcal{P}(x^{k+1}) \nonumber \\
	&\le  \tau \B_{\varphi}(\hat{y},y^{k}) + \mu\B_{\psi}(\hat{x},x^{k}) + \sigma   \mathcal{P}(x^{k}) - c_{1} \B_{\psi}(x^{k+1},x^{k})  \nn \\
	&\quad-  c_{2}\B_{\phi}(y^{k+1},\tilde{y}^{k+1}) - c_{3} \B_{\phi}(\tilde{y}^{k+1},y^{k}), 
	\end{align*}
	which can be rewritten as 
	\begin{equation}\label{eq19}
	\mathcal{D}(\tilde{y}^{k+1}) + \mathcal{P}(x^{k+1}) + {\bm a}_{k+1} \le {\bm a}_{k} - {\bm b}_{k}
	\end{equation}
	with the notations
	\begin{equation}\label{eq20}
	\left\{
	\begin{aligned}
	&{\bm a}_{k} := \tau \B_{\varphi}(\hat{y},y^{k}) + \mu\B_{\psi}(\hat{x},x^{k}) + \sigma \mathcal{P}(x^{k}), \\
	&{\bm b}_{k} :=  c_{1} \B_{\psi}(x^{k+1},x^{k}) +  c_{2}\B_{\phi}(y^{k+1},\tilde{y}^{k+1}) + c_{3} \B_{\phi}(\tilde{y}^{k+1},y^{k})  .
	\end{aligned}\right.
	\end{equation}
	By recalling the updating schemes of $\tilde{y}^{k+1}$ and $x^{k+1}$, it is clear from the definitions of $\mathcal{D}(\cdot) $ and $\mathcal{P}(\cdot) $ that
$\mathcal{D}(\tilde{y}^{k+1}) \ge 0$ and $\mathcal{P}(x^{k+1}) \ge 0$. Moreover, using the definition of Bregman distance yields ${\bm a}_{k} \ge 0$ and ${\bm b}_{k} \ge 0$. Therefore, it follows from \eqref{eq19} that
\begin{equation*}
{\bm a}_{k+1} \le {\bm a}_{k} - {\bm b}_{k}\leq {\bm a}_0 -\sum_{i=0}^{k} {\bm b}_i\leq {\bm a}_0,
\end{equation*}
which further implies that
	\begin{equation}\label{eq16}
	\lim_{k \rightarrow \infty}\B_{\phi}(\tilde y^{k+1},y^{k}) =\lim_{k \rightarrow \infty}\B_{\psi}(x^{k+1},x^{k}) =\lim_{k \rightarrow \infty}\B_{\phi}(y^{k+1},\tilde y^{k+1})  =0,
	\end{equation}
	and the sequence $\{(x^k,\tilde{y}^k,y^k)\}$ is bounded. Moreover,  $\{{\bm a}_k\}$ is also a bounded and monotonically decreasing sequence.
	
	Let $\{(x^{k_{j}},\tilde{y}^{k_j},y^{k_{j}})\}$ be a subsequence converging to a cluster point $(x^\infty,y^\infty,y^\infty)$. 
	It follows from \eqref{eq1} and \eqref{eq3} that
	\begin{equation*}\label{eq17}
	\left\{
	\begin{aligned}
	&f(x)-f(x^{k_j+1})+\left\langle  \mu \left(\nabla \psi(x^{k_j+1}) - \nabla \psi(x^{k_j})\right)+A^\top\tilde{y}^{k_j+1},x-x^{k_j+1} \right\rangle\geq 0,  \\
	&g(y)-g(y^{k_j+1})+\left\langle \gamma \left(\nabla\phi(y^{k_j+1}) - \nabla\phi(y^{k_j})\right)-Ax^{k_j+1},y-y^{k_j+1}\right\rangle \geq 0 
	\end{aligned}\right.
	\end{equation*}
hold for all  $x \in \mathbb{X}$ and $y \in \mathbb{Y}$. By taking $j \rightarrow \infty$ in the above inequalities, it holds, for all  $x \in \mathbb{X}$ and $y \in \mathbb{Y} $, that
	\begin{equation*}
	\left\{
	\begin{aligned}
	&f(x)-f(x^{\infty})+\left\langle  A^\top y^{\infty},x-x^{\infty} \right\rangle\geq 0, \\
	&g(y)-g(y^{\infty})+\left\langle -Ax^{\infty},y-y^{\infty}\right\rangle \geq 0 ,
	\end{aligned}\right.
	\end{equation*}
which, together with the optimality condition in \eqref{KKT}, that $(x^{\infty},y^{\infty})$ is a saddle point of \eqref{sdp}. Using the boundedness and monotonically decreasing property of the sequence $\{{\bm a}_{k}\}$ implies that
	\begin{equation*}\label{eq21}
	\lim_{k \rightarrow \infty}{\bm a}_{k} =\lim_{k \rightarrow \infty}{\bm a}_{k_{j}} = 0. 
	\end{equation*}
	Therefore, we have $\lim_{k \rightarrow \infty}x^{k}=x^{\infty}$ and $\lim_{k \rightarrow \infty}y^{k}= y^{\infty}$, which means that the whole sequence $\{(x^k,y^k)\}$ converges to a solution of \eqref{sdp}. The proof is complete.
\end{proof}


\begin{remark}\label{remark3}
	We notice that the inequality \eqref{eqs} is used as an assumption for the results. Here, we employ an illustrative case to show that there exist positive constants $c_{1}$, $c_{2}$ and $c_{3}$ so that \eqref{eqs} is indeed a reasonable inequality in practice. 
	
	When taking the Bregman kernel functions as $\psi(x) =\frac{1}{2}\| x \|^{2}$ and $\phi(y) = \varphi(y) = \frac{1}{2}\| y \|^{2}$, the left-hand side of inequality \eqref{eqs} can be specified as
		\begin{align}\label{con1}
		\Gamma^{k+1}:=&  \frac{\gamma}{2}  \|y^{k+1}-\tilde{y}^{k+1}\|^{2} + \frac{\gamma}{2} \|\tilde{y}^{k+1}-y^{k}\|^{2} + \left(\frac{\tau}{2} - \frac{\gamma}{2}\right) \|y^{k+1}-y^{k}\|^{2}   \\ 
		& + \frac{\mu(1+2\sigma)}{2} \|x^{k+1}-x^{k}\|^{2}   -(1+\sigma)\langle  A (x^{k+1} - x^{k}) ,y^{k+1} - \tilde{y}^{k+1} \rangle.\nn
		\end{align}
		It follows from the Cauchy–Schwarz and Young's inequalities that for any $\alpha \in (0,+\infty)$, 
		\begin{equation*} 
			(1+\sigma)\langle  A (x^{k+1} - x^{k}) ,y^{k+1} - \tilde{y}^{k+1} \rangle \le \frac{ \alpha}{2} \| x^{k+1} - x^{k} \|^{2} + \frac{\|A^{\top}A\| \left( 1+\sigma \right)^{2}}{2\alpha} \| y^{k+1} - \tilde{y}^{k+1} \|^{2},
		\end{equation*}
		which together with \eqref{con1} leads to
		\begin{align}\label{con2}
			\Gamma^{k+1}  \ge & \left( \frac{\gamma}{2} - \frac{\|A^{\top}A\| \left( 1+\sigma \right)^{2}}{2\alpha} \right)   \|y^{k+1}-\tilde{y}^{k+1}\|^{2} + \frac{\gamma}{2} \|\tilde{y}^{k+1}-y^{k}\|^{2} \nn \\
			&- \left( \frac{\gamma}{2} - \frac{\tau}{2} \right) \|y^{k+1}-y^{k}\|^{2} + \frac{\mu(1+2\sigma) - \alpha}{2} \|x^{k+1}-x^{k}\|^{2}.
		\end{align}
		By the basic inequality $-(a+b)^2 \ge  -2a^{2}-2b^{2}$ for all $a,b \in \R$, we have 
		\begin{align}\label{eq:rem1-1}
			-\|y^{k+1}-y^{k}\|^{2}&=-\|y^{k+1}-\tilde{y}^{k+1}+\tilde{y}^{k+1}-y^{k}\|^{2} \nn\\
			&\geq -\left(\|y^{k+1}-\tilde{y}^{k+1}\|+\|\tilde{y}^{k+1}-y^{k}\|\right)^{2} \nn \\
			&\geq -2\|y^{k+1}-\tilde{y}^{k+1}\|^2-2\|\tilde{y}^{k+1}-y^{k}\|^{2}.
		\end{align}
		Consequently, plugging \eqref{eq:rem1-1} into \eqref{con2} leads to		\begin{align}\label{eq:rem1-2}
			\Gamma^{k+1}  \ge  &\left( \tau -\frac{\gamma}{2} - \frac{\|A^{\top}A\| \left( 1+\sigma \right)^{2}}{2\alpha} \right)   \|y^{k+1}-\tilde{y}^{k+1}\|^{2} +  \left( \tau-\frac{\gamma}{2} \right) \|\tilde{y}^{k+1}-y^{k}\|^{2} \nn \\
			&+ \frac{\mu(1+2\sigma) - \alpha}{2} \|x^{k+1}-x^{k}\|^{2} \nn \\
			 =& \frac{c_{1}}{2} \|x^{k+1}-x^{k}\|^{2} +  \frac{c_{2}}{2}\|y^{k+1}-\tilde{y}^{k+1}\|^{2} + \frac{c_{3}}{2} \|\tilde{y}^{k+1}-y^{k}\|^{2},
		\end{align}
		where $c_{1} := \mu(1+2\sigma) - \alpha$, $c_{2} := 2\tau - \gamma -\frac{\|A^{\top}A\| \left( 1+\sigma \right)^{2}}{\alpha} $ and $c_{3} := 2\tau - \gamma$. 
		
		Below, we discuss the existence of $\mu$, $\tau$, and $\gamma$ such that the right-hand side of \eqref{eq:rem1-2} is positive. Under the settings mentioned at the begining of this remark, we notice that both subproblems \eqref{a} and \eqref{d} share the same Bregman kernel function. Therefore, without loss of generality, we below let $ \tau = \theta \gamma$ for the purpose of simplifying our analysis, where $\theta$ is required to be $\theta > \frac{1}{2}$. Now, we divide two cases to discuss the values of $\theta$.

		{\bf Case I}: When $\frac{1}{2}<\theta < 1$, we first have
		\begin{align}\label{c5}
		-\left(\frac{\gamma}{2}- \frac{\tau}{2}\right) \|y^{k+1}-y^{k}\|^{2}& = - \frac{1-\theta}{2}\gamma  \|y^{k+1}-y^{k}\|^{2} \nn  \\
		&\ge - (1-\theta) \gamma   \|y^{k+1}-\tilde{y}^{k+1}\|^{2} -(1-\theta) \gamma   \|\tilde{y}^{k+1}-y^{k}\|^{2}.
		\end{align}
		By plugging \eqref{c5} into \eqref{con2}, we then arrive at
		\begin{align*} 
		&\left( \frac{\gamma(2\theta-1)}{2}-  \frac{\|A^{\top}A\| \left( 1+\sigma \right)^{2}}{2\alpha} \right)  \|y^{k+1}-\tilde{y}^{k+1}\|^{2} \nn \\ 
	&\qquad	+ \frac{\gamma(2\theta-1)}{2}   \|\tilde{y}^{k+1}-y^{k}\|^{2} + \frac{\mu(1+2\sigma)-\alpha}{2} \|x^{k+1}-x^{k}\|^{2}  >  0,
		\end{align*}
		which holds for $\sigma\geq 0$ under the requirement:
		\begin{equation*} 
			\mu(1+2\sigma) > \alpha > \frac{(1 + \sigma)^{2}}{\gamma(2\theta-1)} \|A^\top A\|.
		\end{equation*}
		As a result, we get the condition:
		\begin{equation}\label{eq:rem1-3} 
		\mu \gamma > \frac{(1 + \sigma)^{2}}{(1 + 2 \sigma)(2\theta -1)} \|A^\top A\| > \|A^\top A\|,\quad \forall \sigma\geq 0.
		\end{equation}
		
		{\bf Case II}: When $ \theta \ge 1$, it follows that
		\begin{align}\label{zb}
			\frac{\gamma}{2} \|\tilde{y}^{k+1}-y^{k}\|^{2} + \left(\frac{\tau}{2} - \frac{\gamma}{2}\right) \|y^{k+1}-y^{k}\|^{2}  
			&=  \frac{\gamma}{2} \|\tilde{y}^{k+1}-y^{k}\|^{2} + \frac{(\theta-1)\gamma}{2}  \|y^{k+1}-y^{k}\|^{2} \nn \\
			&\geq \frac{c\gamma}{2}\left( \|\tilde{y}^{k+1}-y^{k}\|^{2}+\|y^{k+1}-y^{k}\|^{2}\right) \nn \\
			&\geq \frac{c\gamma}{4}\left( \|\tilde{y}^{k+1}-y^{k}\|+\|y^{k+1}-y^{k}\|\right)^{2} \nn \\
			& \ge \frac{c\gamma}{4 } \| y^{k+1} - \tilde{y}^{k+1} \|^{2},
		\end{align}
		where the first inequality holds for the definition $c: = \min\{ 1,\theta-1 \}$, the second inequality follows from the fact that $ a^{2}+b^{2}\geq \frac{(a+b)^2}{2}$ for all $a,b \in \R$, and the last inequality is due to the triangle inequality.
		
		Hereafter, by the definition of $c$ in \eqref{zb}, we further divide $\theta \geq 1$ into two cases, i.e., $1\leq \theta<2$ and $\theta\geq 2$ for discussions.
		\begin{itemize}
			\item 	When $1 \le \theta < 2$, we immediately have $c = \theta-1 < 1 $.  Therefore, by plugging \eqref{zb} into \eqref{con2}, we have 
			\begin{align*}\label{d1}
				\left( \frac{(\theta+1)\gamma}{4} - \frac{\|A^{\top}A\| \left( 1+\sigma \right)^{2}}{2\alpha} \right)   \|y^{k+1}-\tilde{y}^{k+1}\|^{2} + \frac{\mu(1+2\sigma) - \alpha}{2} \|x^{k+1}-x^{k}\|^{2}  > 0,
			\end{align*}
			which holds under the following condition: 
			\begin{equation*}\label{d2}
				\mu(1+2\sigma) > \alpha > \frac{2(1 + \sigma)^{2}}{(\theta+1)\gamma} \|A^\top A\|.
			\end{equation*}		
			As an immediate result, we get the condition
			\begin{equation}\label{c2}
				\mu \gamma > \frac{2(1 + \sigma)^{2}}{(\theta+1)(1 + 2 \sigma)} \|A^\top A\| >   \frac{2}{3}\|A^\top A\|,\quad \forall \sigma\geq 0.
			\end{equation}
			
			\item 	When $\theta \ge 2$, then the parameter $c$ is specified as $c = 1$. Similarly, by following the derivation of \eqref{c2}, we have
			\begin{equation}\label{c1}
				\mu \gamma > \frac{2(1 + \sigma)^{2}}{3 + 6 \sigma} \|A^\top A\| \geq  \frac{2}{3}\|A^\top A\|,\quad \forall \sigma\geq 0.
			\end{equation}	  
		\end{itemize}
		With the above discussions, we conclude from \eqref{eq:rem1-3}, \eqref{c2}, and \eqref{c1} that, under the setting $\tau=\theta \gamma$ and $\sigma\geq 0$, the three parameters $\mu$, $\tau$, and $\gamma$ satisfy
        \begin{equation}\label{c6}
          \left\{ \begin{aligned}
&\mu \gamma > \frac{(1 + \sigma)^{2}}{(1 + 2 \sigma)(2\theta -1)} \|A^\top A\|>  \|A^\top A\|, \, \; \text{if}\;\; \frac{1}{2} < \theta < 1, \\
		&\mu \gamma > \frac{2(1 + \sigma)^{2}}{(\theta+1)(1 + 2 \sigma)} \|A^\top A\|>\frac{2}{3} \|A^\top A\|, \, \; \text{if}\;\; 1 \le \theta <  2, \\
		&\mu \gamma >\frac{2(1 + \sigma)^{2}}{3 + 6 \sigma} \|A^\top A\|\geq \frac{2}{3} \|A^\top A\|, \, \; \text{if}\;\;  \theta \ge 2.
\end{aligned}\right.
\end{equation}

As a result, it can be easily seen form \eqref{c6} that our TBDA allows larger stepsizes than the standard PDHG \cite{CP11} when $\theta\geq 1$ (see condition \eqref{condition}). Moreover, our TBDA enjoys a larger range of $\mu$ and $\gamma$ than the results in \cite{HMXY22,HWY23,JZH23,LY21} when $\theta\geq 1$, which also be beneficial for numerical improvements.
\end{remark}

\begin{theorem}
	Let $\{(\hat{x},\hat{y})\}$ be a saddle point of \eqref{sdp}, and let $\{(x^{k},\tilde y^{k},{y}^{k})\}$ be a sequence generated by Algorithm \ref{alg1} from any initial point $(x^0, y^0)$. Then, for all $x \in \mathbb{X}$ and $y \in \mathbb{Y}$, it holds that 
	\begin{equation}\label{eq22}
	\mathcal{G}(\mathbf{x}^{N},\mathbf{y}^{N}) \le  \frac{1}{N} \left(\mu\B_{\psi}(\hat{x},x^{0}) + \tau \B_{\varphi}(\hat{y},y^{0}) + \sigma \mathcal{P}(x^{0})\right),
	\end{equation}
	where $N$ is a positive integer, $\mathbf{x}^{N} := \frac{1}{\sigma + N}\left(\sigma x^{N} + \sum_{k=1}^{N} x^{k} \right)$ and $\mathbf{y}^{N} := \frac{1}{N}\sum_{k=1}^{N} \tilde{y}^{k}$.
\end{theorem}

\begin{proof}
	Summing up \eqref{eq12} from $0$ to $N-1$, we get
	\begin{align}\label{eq23}
	&\sum_{k=0}^{N-1} \left\{\mathcal{D}(\tilde{y}^{k+1}) + (1+\sigma)\mathcal{P}(x^{k+1}) - \sigma  \mathcal{P}(x^{k})\right\}  \\
	&\le \mu(\B_{\psi}(\hat{x},x^{0}) - \B_{\psi}(\hat{x},x^{N}))  + \tau(\B_{\varphi}(\hat{y},y^{0}) - \B_{\varphi}(\hat{y},y^{N})) \nonumber \\
	&\le  \mu\B_{\psi}(\hat{x},x^{0}) + \tau\B_{\varphi}(\hat{y},y^{0}). \nonumber 
	\end{align}
	Since  $\mathcal{P}(x)$ and $\mathcal{D}(y)$ are convex with respect to $x$ and $y$, respectively, it follows from the left-hand side of \eqref{eq23} and the Jesen inequality that 
	\begin{align}\label{eq24}
	&\sum_{k=0}^{N-1} \left\{\mathcal{D}(\tilde{y}^{k+1}) + (1+\sigma)\mathcal{P}(x^{k+1}) - \sigma  \mathcal{P}(x^{k})\right\} \nonumber \\
	& = \sum_{k=0}^{N-1}\left\{ \mathcal{D}(\tilde{y}^{k+1}) + \sigma (\mathcal{P}(x^{k+1}) -  \mathcal{P}(x^{k})) + \mathcal{P}(x^{k+1})\right\} \nonumber \\
	&=  - \sigma \mathcal{P}(x^{0}) + \sigma \mathcal{P}(x^{N})  + \sum_{k=1}^{N} \left\{\mathcal{D}(\tilde{y}^{k}) + \mathcal{P}(x^{k})\right\} \nonumber \\
 &	\ge (\sigma + N) \mathcal{P}\left(\frac{\sigma x^{N} + \sum_{k=1}^{N} x^{k} }{\sigma + N}\right) + N \mathcal{D}\left(\frac{\sum_{k=1}^{N} \tilde{y}^{k}}{N}\right) - \sigma \mathcal{P}(x^{0}) \nonumber \\
	 &\ge  N \left(\mathcal{P}\left(\frac{\sigma x^{N} + \sum_{k=1}^{N} x^{k} }{\sigma + N}\right) +  \mathcal{D}\left(\frac{\sum_{k=1}^{N} \tilde{y}^{k}}{N}\right) \right) - \sigma \mathcal{P}(x^{0}) \nonumber \\
	& =  N (\mathcal{P}(\mathbf{x}^{N}) +  \mathcal{D}(\mathbf{y}^{N})) - \sigma \mathcal{P}(x^{0}).
	\end{align}
	Substituting \eqref{eq24} into \eqref{eq23}, we conclude that
	\begin{equation*}
	\mathcal{G}(\mathbf{x}^{N},\mathbf{y}^{N}) = \mathcal{P}(\mathbf{x}^{N}) +  \mathcal{D}(\mathbf{y}^{N}) \le  \frac{1}{N}\left(\mu\B_{\psi}(\hat{x},x^{0}) + \tau\B_{\varphi}(\hat{y},y^{0}) + \sigma \mathcal{P}(x^{0})\right).
	\end{equation*}
	The proof is complete.
\end{proof}

\begin{remark}
	The result in \eqref{eq22} means that our Algorithm \ref{alg1} has the $O(1/N)$ convergence rate.
\end{remark}

\section{Improved Variants of TBDA}\label{Sec4}
We have shown that Algorithm \ref{alg1} has the $O(1/N)$ convergence rate under some mild conditions. Therefore, we are further concerned with the question: Can we improve the convergence rate of Algorithm \ref{alg1} as the PDHG \cite{CP11} when strengthening the objective functions $f$ and/or $g$? In this section, we give an affirmative answer to this question by introducing two improved variants.

\subsection{When $f$ is strongly convex relative to $\psi$} 

From the discussion in Remark \ref{remark3}, we observe that the parameter $\theta$ in $\tau=\theta\gamma$ is very important for enlarging the range of stepsizes of Algorithm \ref{alg1}. Therefore, we are motivated to modify the second $y$-subproblem of Algorithm \ref{alg1} for algorithmic acceleration. To maximally keep the original form of Algorithm \ref{alg1}, we only specify $\tau=\beta_k \gamma$ with a dynamical $\beta_k$ to modify the iterative scheme of $y^{k+1}$. More specifically, we present the first improved variant of Algorithm \ref{alg1} for \eqref{sdp} with $\varrho_{1}$-strongly convex $f$ relative to the Bregman kernel $\psi$ in Algorithm \ref{alg2}.

\begin{algorithm}[!h]
	\caption{Improved TBDA for \eqref{sdp} with relative $\varrho_{1}$-strongly convex $f$.}\label{alg2}
	\begin{algorithmic}[1]
		\STATE Choose starting points $(x^{0},y^{0}) \in \mathbb{X}\times \mathbb{Y}$ and appropriate parameters $\gamma>0$, $\mu>0$, $\tau>0$, $\sigma \ge 0$, $\beta_{0}>0$ and $0<p<2$.
		\REPEAT 
		\STATE  Update $\tilde{y}^{k+1}$, $x^{k+1}$, and $\bar{x}^{k+1}$ via \eqref{a}, \eqref{b}, and \eqref{c}, respectively.
		\STATE  Update $y^{k+1}$ via 
		\begin{align}
		&y^{k+1} = \arg\min_{y\in \mathbb{Y}}\left\{ g(y) - \langle A \bar{x}^{k+1}, y\rangle +  \gamma \beta_{k}  \B_{\varphi}(y,y^{k}) \right\} \label{e}. 	\\
		&\beta_{k+1}  = \min\left\{ \frac{\mu\beta_{k}}{\mu + \varrho_{1}}, \frac{1}{p}  \right\}. \label{ubeta}
		\end{align} 
		\UNTIL some stopping criterion is satisfied.
		\RETURN an approximate saddle point $(\hat{x},\hat{y})$.
	\end{algorithmic}
\end{algorithm}

Below, we establish the similar result to Theorem \ref{th2} for Algorithm \ref{alg2}.
\begin{theorem}\label{thm3}
	Let $\{(\hat{x},\hat{y})\}$ be a saddle point of \eqref{sdp}, and let $\{(\tilde y^{k},x^{k}, {y}^{k}) \}$ be a sequence generated by Algorithm \ref{alg2} from any initial points $(x^0, y^0)$. If there exist positive constants $c_{1}$, $c_{2}$, $c_{3}$  such that  
	\begin{align*}
	& \mu(1+\sigma) \B_{\psi}(x^{k+1},x^{k})+( \sigma \mu + \sigma \varrho_{1} )\B_{\psi}(x^{k},x^{k+1})  + \gamma  \B_{\phi}(y^{k+1},\tilde{y}^{k+1})     \nn  \\
	&\quad + \gamma \B_{\phi}(\tilde{y}^{k+1},y^{k}) +\beta_{k} \gamma\B_{\varphi}(y^{k+1},y^{k}) - \gamma \B_{\phi}(y^{k+1},y^{k}) \nn \\ 
	&\quad -(1+\sigma)\langle  A (x^{k+1} - x^{k}) ,y^{k+1} - \tilde{y}^{k+1} \rangle \nn \\
	&\ge  c_{1} \B_{\psi}(x^{k+1},x^{k}) +  c_{2}\B_{\phi}(y^{k+1},\tilde{y}^{k+1}) + c_{3} \B_{\phi}(\tilde{y}^{k+1},y^{k}).
	\end{align*}
	Then, we have 
	\begin{equation}\label{eq26}
	\mathcal{G}(\mathbf{x}^{N},\mathbf{y}^{N}) \le  \frac{1}{t_{N}} \left( \beta_{0}\gamma  \B_{\varphi}(\hat{y},y^{0}) + \mu\B_{\psi}(\hat{x},x^{0}) + \frac{1}{\beta_{0}}   \sigma  \mathcal{P}(x^{0})\right),
	\end{equation}
	where $N$ is a positive integer, $t_{N} :=  \sum_{k=1}^{N} \frac{1}{\beta_{k-1}}  $, 
	$$\mathbf{x}^{N}: = \frac{  \sigma x^{0} + \beta_{0} \sum_{k=1}^{N} \frac{1}{\beta_{k-1}} \bar{x}^{k} }{  \sigma + \beta_{0}\sum_{k=1}^{N} \frac{1}{\beta_{k-1}} } ,
	\quad \text{and}\quad \mathbf{y}^{N} := \frac{\sum_{k=1}^{N} \frac{1}{\beta_{k-1}} \tilde{y}^{k}}{\sum_{k=1}^{N} \frac{1}{\beta_{k-1}}}.$$
\end{theorem}

\begin{proof}
	First, by invoking the setting $\tau=\beta_k\gamma$ in \eqref{d} (i.e., \eqref{e}), the inequality \eqref{eq2} becomes
	\begin{equation}\label{eq28}
	\langle\beta_{k} \gamma (\nabla \varphi(y^{k+1}) - \nabla \varphi(y^{k})) -  A \bar{x}^{k+1} ,y-y^{k+1} \rangle \ge g(y^{k+1}) - g(y), \; \forall y \in \mathbb{Y}.
	\end{equation}
	Then, we know from the $\varrho_{1}$-strong convexity of $f$ relative to the Bregman kernel $\psi$ and \eqref{b} that the inequality \eqref{eq3} implies 
	\begin{align}\label{eq29}
	&\langle \mu ( \nabla \psi(x^{k+1}) - \nabla \psi(x^{k})) +  A^\top \tilde{y}^{k+1} ,x- x^{k+1} \rangle \\
	&\hskip 3cm \ge f(x^{k+1}) - f(x) + \varrho_{1} \B_{\psi}(x,x^{k+1}). \nn
	\end{align}
	Based on Lemma \ref{le1}, combining \eqref{eq1}, \eqref{eq28} and \eqref{eq29} yields
	\begin{align}\label{eq31}
	&\beta_{k} \gamma \B_{\varphi}(\hat{y},y^{k+1}) + (\mu + \varrho_{1})\B_{\psi}(\hat{x},x^{k+1}) + \mathcal{D}(\tilde{y}^{k+1}) +  (\sigma +1) \mathcal{P}(x^{k+1}) \nonumber \\
&	\le \beta_{k} \gamma \B_{\varphi}(\hat{y},y^{k}) + \mu\B_{\psi}(\hat{x},x^{k}) + \sigma \mathcal{P}(x^{k})  -\beta_{k} \gamma\B_{\varphi}(y^{k+1},y^{k})  \nonumber \\  
	&\quad  - ( \sigma \mu + \sigma \varrho_{1} )\B_{\psi}(x^{k},x^{k+1}) - \mu(1+\sigma) \B_{\psi}(x^{k+1},x^{k}) - \gamma  \B_{\phi}(y^{k+1},\tilde{y}^{k+1})   \nonumber \\
	&\quad   + (1+\sigma)\langle  A (x^{k+1} - x^{k}) ,y^{k+1}- \tilde{y}^{k+1} \rangle - \gamma \B_{\phi}(\tilde{y}^{k+1},y^{k}) + \gamma \B_{\phi}(y^{k+1},y^{k}) \nonumber \\
	&\le \beta_{k} \gamma \B_{\varphi}(\hat{y},y^{k}) + \mu\B_{\psi}(\hat{x},x^{k}) + \sigma   \mathcal{P}(x^{k}).
	\end{align}
	By using \eqref{ubeta}, the first two terms on the left-hand side of \eqref{eq31} implies that
	\begin{align}\label{eq32}
	&\beta_{k} \gamma \B_{\varphi}(\hat{y},y^{k+1}) + (\mu + \varrho_{1})\B_{\psi}(\hat{x},x^{k+1})   \nonumber \\ 
	&=  \frac{\beta_{k}}{\beta_{k+1}}\beta_{k+1}\gamma \B_{\varphi}(\hat{y},y^{k+1}) + \frac{\beta_{k}}{\beta_{k+1}}\mu\B_{\psi}(\hat{x},x^{k+1})  \nonumber \\ 
	&=  \frac{\beta_{k}}{\beta_{k+1}}\left( \beta_{k+1}\gamma \B_{\varphi}(\hat{y},y^{k+1}) + \mu\B_{\psi}(\hat{x},x^{k+1})\right)  \nn \\
	&\le  \beta_{k} \gamma \B_{\varphi}(\hat{y},y^{k}) + \mu\B_{\psi}(\hat{x},x^{k}) + \sigma  \mathcal{P}(x^{k}) -  \mathcal{D}(\tilde{y}^{k+1}) -  (\sigma + 1) \mathcal{P}(x^{k+1}).
	\end{align} 
	To simplify the notation, we let 
	\begin{equation*}
	\left\{
	\begin{aligned}
	&\widehat{\bm{a}}_{k} := \beta_{k}\gamma \B_{\varphi}(\hat{y},y^{k}) + \mu\B_{\psi}(\hat{x},x^{k}) , \\
	&\widehat{\bm{c}}_{k} :=\sigma  \mathcal{P}(x^{k}) -  \mathcal{D}(\tilde{y}^{k+1}) -  (\sigma + 1) \mathcal{P}(x^{k+1}) .
	\end{aligned}\right.
	\end{equation*} 
	Then, \eqref{eq32} can be simplified as
	\begin{equation*}
	\frac{\beta_{k}}{\beta_{k+1}}\widehat{\bm{a}}_{k+1} \le \widehat{\bm{a}}_{k} + \widehat{\bm{c}}_{k}, 
	\end{equation*}
	or equivalently,
	\begin{equation}\label{eq33}
	\frac{1}{\beta_{k+1}} \widehat{\bm{a}}_{k+1} - \frac{1}{\beta_{k}} \widehat{\bm{c}}_{k} \le \frac{1}{\beta_{k}} \widehat{\bm{a}}_{k}. 
	\end{equation}
	Now, summing up \eqref{eq33} from $0$ to $N-1$ arrives at
	\begin{equation}\label{eq35}
	\frac{1}{\beta_{N}} \widehat{\bm{a}}_{N} - \sum_{k=0}^{N-1} \frac{1}{\beta_{k}} \widehat{\bm{c}}_{k} \le \frac{1}{\beta_{0}} \widehat{\bm{a}}_{0}. 
	\end{equation}
	Similar to (\ref{eq24}), by convexity of $\mathcal{P}$ and $\mathcal{D}$, we get
	\begin{align}\label{eq36}
	-\sum_{k=0}^{N-1}  \frac{1}{\beta_{k}} \widehat{\bm{c}}_{k} 
	 &=\sum_{k=0}^{N-1}  \frac{1}{\beta_{k}} \mathcal{D}(\tilde{y}^{k+1}) +  \frac{1}{\beta_{k}} (1+\sigma)\mathcal{P}(x^{k+1}) -  \frac{1}{\beta_{k}} \sigma  \mathcal{P}(x^{k}) \nonumber \\
	&=  -  \frac{1}{\beta_{0}} \sigma  \mathcal{P}(x^{0}) +  \frac{1}{\beta_{N-1}} (1+\sigma)\mathcal{P}(x^{N}) + \sum_{k=0}^{N-1}  \frac{1}{\beta_{k}} \mathcal{D}(\tilde{y}^{k+1}) \nn \\
	&\quad + \sum_{k=1}^{N-1}  \left( \frac{1}{\beta_{k-1}}  (1+\sigma) - \frac{1}{\beta_{k}} \sigma\right) \mathcal{P}(x^{k})  \nonumber \\
	&\ge  \left( \frac{1}{\beta_{0}} \sigma + \sum_{k=0}^{N-1}  \frac{1}{\beta_{k}} \right)  \mathcal{P}\left(\frac{\frac{1}{\beta_{0}} \sigma x^{0} + \sum_{k=0}^{N-1}  \frac{1}{\beta_{k}} \bar{x}^{k+1} }{\frac{1}{\beta_{0}} \sigma + \sum_{k=0}^{N-1} \frac{1}{\beta_{k}} }\right)  \nn \\
	&\quad + \left(\sum_{k=0}^{N-1}  \frac{1}{\beta_{k}}\right) \mathcal{D}\left(\frac{\sum_{k=0}^{N-1}  \frac{1}{\beta_{k}} \tilde{y}^{k+1}}{\sum_{k=0}^{N-1}  \frac{1}{\beta_{k}}}\right)-  \frac{1}{\beta_{0}} \sigma  \mathcal{P}(x^{0})   \nonumber \\
	&\ge  \left( \sum_{k=1}^{N} \frac{1}{\beta_{k-1}} \right)  \left(\mathcal{P}\left(\frac{  \sigma x^{0} + \beta_{0} \sum_{k=1}^{N} \frac{1}{\beta_{k-1}} \bar{x}^{k} }{  \sigma + \beta_{0}\sum_{k=1}^{N} \frac{1}{\beta_{k-1}} }\right) +  \mathcal{D}\left(\frac{\sum_{k=1}^{N} \frac{1}{\beta_{k-1}} \tilde{y}^{k}}{\sum_{k=1}^{N} \frac{1}{\beta_{k-1}}}\right)\right)        \nonumber \\
	&\quad -  \frac{1}{\beta_{0}} \sigma  \mathcal{P}(x^{0})  \nn \\
	&=  t_{N}  (\mathcal{P}(\mathbf{x}^{N}) +  \mathcal{D}(\mathbf{y}^{N})) -  \frac{1}{\beta_{0}} \sigma  \mathcal{P}(x^{0}).
	\end{align}
	Consequently, substituting \eqref{eq36} into \eqref{eq35} yields
	\begin{equation*}
	\mathcal{G}(\mathbf{x}^{N},\mathbf{y}^{N}) = \mathcal{P}(\mathbf{x}^{N}) +  \mathcal{D}(\mathbf{y}^{N}) \le  \frac{1}{t_{N}} \left( \beta_{0} \gamma \B_{\varphi}(\hat{y},y^{0}) + \mu\B_{\psi}(\hat{x},x^{0}) +  \frac{1}{\beta_{0}} \sigma  \mathcal{P}(x^{0})\right).
	\end{equation*}
The proof is complete.
\end{proof}

\begin{remark}
		Since $\frac{1}{\beta_{k+1}}  = \frac{\mu+\varrho_{1}}{\mu}   \frac{1}{\beta_{k}} $ and $\frac{1}{\beta_{k}} < 2$, we can conclude through induction that the value of $\frac{1}{\beta_{k}}$ does not change after a finite number of $k$ iterations. Thus, $t_{N} = \sum_{k=1}^{N} \frac{1}{\beta_{k+1}} = O(N)$. Then, it follow from \eqref{eq26} that there exists some constant $ C > 0$ such that
	\begin{equation*}
	\mathcal{G}(\mathbf{x}^{N},\mathbf{y}^{N}) \le  \frac{C}{N},
	\end{equation*} 
	which means that Algorithm \ref{alg2} also enjoys $O(1/N)$ convergence rate. In this situation, it is possible to obtain a tighter upper bound than Theorem \ref{th2} under an appropriate setting on the $\beta_k$.
\end{remark}

\subsection{When $f$ and $g$ are strongly convex relative to $\psi$ and $\varphi$, respectively.} 
In this subsection, we first assume that $f$ is $\varrho_1$-strongly convex relative to the Bregman kernel $\psi$ and $g$ is $\varrho_2$-strongly convex relative to the Bregman kernel $\varphi$. Then, with two specifications on $\mu$ and $\tau$, i.e., $\mu=\frac{\varrho_{1}}{\omega -1}$ and $\tau=\frac{\varrho_{2}}{\omega -1}$ with $\omega >1$ being a given constant, we show that Algorithm \ref{alg1} can be accelerated with $O(1/\omega^N)$ convergence rate. We first present details of the accelerated version of TBDA in Algorithm \ref{alg3}.

\begin{algorithm}[!h]
	\caption{Accelerated TBDA for \eqref{sdp} with strongly convex $f$ and  and $g$ relative to $\psi$ and $\varphi$, respectively.}\label{alg3}
	\begin{algorithmic}[1]
		\STATE Choose starting points $(x^{0},y^{0}) \in \mathbb{X}\times  \mathbb{Y}$ and parameters $\gamma >0$ and $\omega >1$. 
		\REPEAT
		\STATE  Update $\tilde{y}^{k+1}$  via \eqref{a}.
		\STATE  Update $x^{k+1}$ via \eqref{b} with setting $\mu = \frac{\varrho_{1}}{\omega-1} $.
		\STATE  Update $\bar{x}^{k+1}$ via \eqref{c}.
		\STATE  Update $y^{k+1}$ via \eqref{d} with setting $\tau = \frac{\varrho_{2}}{\omega-1} $.
		\UNTIL some stopping criterion is satisfied.
		\RETURN an approximate saddle point $(\hat{x},\hat{y})$.
	\end{algorithmic}
\end{algorithm}

Below, we establish the similar result to Theorems \ref{th2} and \ref{thm3} for Algorithm \ref{alg3}.

\begin{theorem}
	Let $\{(\hat{x},\hat{y})\}$ be a saddle point of \eqref{sdp}, and let $\{(\tilde y^{k+1},x^{k+1},{y}^{k+1}) \}$ be a sequence generated by Algorithm \ref{alg3} from any initial point $(x^0, y^0)$. If there exist positive constants $c_{1}$, $c_{2}$, $c_{3}$  such that   
	\begin{align*}\label{dsc}
	& \mu(1+\sigma) \B_{\psi}(x^{k+1},x^{k})+  (\sigma \mu + \varrho_{1} \sigma) \B_{\psi}(x^{k},x^{k+1}) + \gamma \B_{\phi}(y^{k+1},\tilde{y}^{k+1})  \nn  \\
	&\quad + \varrho_{2} \B_{\varphi}(y^{k+1},\tilde{y}^{k+1}) + \gamma \B_{\phi}(\tilde{y}^{k+1},y^{k}) + \tau\B_{\varphi}(y^{k+1},y^{k}) - \gamma \B_{\phi}(y^{k+1},y^{k}) \nn \\ 
	&\quad -(1+\sigma)\langle  A (x^{k+1} - x^{k}) ,y^{k+1} - \tilde{y}^{k+1} \rangle \nn \\
	&\ge  c_{1} \B_{\psi}(x^{k+1},x^{k}) +  c_{2}\B_{\phi}(y^{k+1},\tilde{y}^{k+1}) + c_{3} \B_{\phi}(\tilde{y}^{k+1},y^{k}).
	\end{align*}
	Then, we have
	\begin{equation}\label{eq27}
	\mathcal{G}(\mathbf{x}^{N},\mathbf{y}^{N})  \le  \frac{1}{\omega^{N-2}}\left( \tau  \B_{\varphi}(\hat{y},y^{0}) +  \mu \B_{\psi}(\hat{x},x^{0}) +    \sigma  \mathcal{P}(x^{0})\right).
	\end{equation}
	where $N$ is a positive integer, 
	$$\mathbf{x}^{N} := \frac{ \sigma x^{0} + \sum_{k=1}^{N} \omega^{k-1} \bar{x}^{k} }{  \sigma + \sum_{k=1}^{N} \omega^{k-1}} \quad \text{and}\quad \mathbf{y}^{N} := \frac{\sum_{k=1}^{N} \omega^{k-1} \tilde{y}^{k}}{\sum_{k=1}^{N} \omega^{k-1}}$$ with a given constant $\omega>1$.
\end{theorem}

\begin{proof}
	Similar to \eqref{eq1}, \eqref{eq2} and \eqref{eq3}, under the conditions that $f$ is $\varrho_1$-strongly convex relative to $\psi$ and $g$ is $\varrho_2$-strongly convex relative to $\varphi$, it follows from the iterative schemes of Algorithm \ref{alg3} that 
	\begin{equation}\label{eq37}
	\left\{ \begin{aligned}
	& \langle \gamma \left(\nabla \phi(\tilde{y}^{k+1}) - \nabla \phi(y^{k})\right) -  A x^{k} ,y-\tilde{y}^{k+1} \rangle \\
	&\hskip 5cm \ge g(\tilde{y}^{k+1}) - g(y) + \varrho_{2} \B_{\varphi}(y,\tilde{y}^{k+1}), \\
	&\langle \tau \left(\nabla \varphi(y^{k+1}) - \nabla \varphi(y^{k})\right) -  A \bar{x}^{k+1} ,y-y^{k+1} \rangle  \\
	&\hskip 5cm \ge g(y^{k+1}) - g(y) +\varrho_{2} \B_{\varphi}(y,y^{k+1}), \\
	&\langle \mu  \left(\nabla \psi(x^{k+1}) - \nabla \psi(x^{k})\right) +  A^\top \tilde{y}^{k+1} ,x- x^{k+1} \rangle \\
	&\hskip 5cm \ge  f(x^{k+1}) - f(x) + \varrho_{1} \B_{\psi}(x,x^{k+1}). 
	\end{aligned}
	\right.
	\end{equation}
	holds for all $y\in\mathbb{Y}$ and $x\in\mathbb{X}$. By following the analysis of Lemma \ref{le1}, it follows from \eqref{eq37} that
	\begin{align*}
	& (\tau + \varrho_{2}) \B_{\varphi}(\hat{y},y^{k+1}) +  (\mu + \varrho_{1})\B_{\psi}(\hat{x},x^{k+1}) + \mathcal{D}(\tilde{y}^{k+1}) + (\sigma +1) \mathcal{P}(x^{k+1}) \nonumber \\
	&=  \tau \B_{\varphi}(\hat{y},y^{k}) + \mu\B_{\psi}(\hat{x},x^{k}) + \sigma   \mathcal{P}(x^{k}) - \tau\B_{\varphi}(y^{k+1},y^{k}) + \gamma \B_{\phi}(y^{k+1},y^{k})    \nonumber \\  
	&\quad - \gamma \B_{\phi}(\tilde{y}^{k+1},y^{k})  - (\sigma \mu + \varrho_{1} \sigma)\B_{\psi}(x^{k},x^{k+1})  - \mu(1+\sigma) \B_{\psi}(x^{k+1},x^{k})     \nonumber \\
	&\quad - \gamma \B_{\phi}(y^{k+1},\tilde{y}^{k+1}) - \varrho_{2} \B_{\varphi}(y^{k+1},\tilde{y}^{k+1}) +(1+\sigma)\langle  A (x^{k+1} - x^{k}) ,y^{k+1}- \tilde{y}^{k+1} \rangle     \nonumber \\
 &	\le  \tau \B_{\varphi}(\hat{y},y^{k}) + \mu\B_{\psi}(\hat{x},x^{k}) + \sigma   \mathcal{P}(x^{k}),
	\end{align*}
	where, by recalling $\mu=\frac{\varrho_{1}}{\omega -1}$ and $\tau=\frac{\varrho_{2}}{\omega -1}$,  its first two terms imply that
	\begin{align}\label{eq42}
	&(\tau + \varrho_{2}) \B_{\varphi}(\hat{y},y^{k+1}) +  (\mu + \varrho_{1})\B_{\psi}(\hat{x},x^{k+1})  \nonumber \\ 
	&=  \omega \left(\tau \B_{\varphi}(\hat{y},y^{k+1}) +  \mu \B_{\psi}(\hat{x},x^{k+1})\right)  \nn \\
	&\le  \left(\tau \B_{\varphi}(\hat{y},y^{k}) + \mu\B_{\psi}(\hat{x},x^{k}) \right)+ \sigma \mathcal{P}(x^{k}) - \mathcal{D}(\tilde{y}^{k+1}) - (1+\sigma)\mathcal{P}(x^{k+1}). 
	\end{align} 
	To simplify the notation, we let 
		\begin{equation*}
	\left\{
	\begin{aligned}
	&\widetilde{\bm{a}}_{k} := \tau \B_{\varphi}(\hat{y},y^{k}) + \mu\B_{\psi}(\hat{x},x^{k}) , \\
	&\widetilde{\bm{c}}_{k} :=\sigma \mathcal{P}(x^{k}) -  \mathcal{D}(\tilde{y}^{k+1}) -  (\sigma + 1) \mathcal{P}(x^{k+1}).
	\end{aligned}\right.
	\end{equation*} 
	As a consequence, by multiplying $\omega^{k}$ on both sides of \eqref{eq42}, we immediately have 
	\begin{equation}\label{eq43}
	\omega^{k+1} \widetilde{\bm{a}}_{k+1} \le \omega^{k} \widetilde{\bm{a}}_{k} +  \omega^{k} \widetilde{\bm{c}}_{k}. 
	\end{equation}
   Hence, summing up \eqref{eq43} from $0$ to $N-1$, we obtain
	\begin{equation}\label{eq44}
	\omega^{N} \widetilde{\bm{a}}_{N} - \sum_{k=0}^{N-1} \omega^{k} \widetilde{\bm{c}}_{k}  \le  \omega^{0}\widetilde{\bm{a}}_{0}\equiv \widetilde{\bm{a}}_{0}.
	\end{equation}
	Similar to \eqref{eq24}, we get
	\begin{align}\label{eq46}
	-\sum_{k=0}^{N-1}  \omega^{k} \widetilde{\bm{c}}_{k} 
	&= \sum_{k=0}^{N-1}  \omega^{k} \mathcal{D}(\tilde{y}^{k+1}) +  \omega^{k} (1+\sigma)\mathcal{P}(x^{k+1}) -  \omega^{k} \sigma  \mathcal{P}(x^{k}) \nonumber \\
	&=  -  \sigma  \mathcal{P}(x^{0}) +  \frac{1}{\omega^{N-1}} (1+\sigma)\mathcal{P}(x^{N}) + \sum_{k=0}^{N-1}  \omega^{k} \mathcal{D}(\tilde{y}^{k+1}) \nn \\
	&\quad + \sum_{k=1}^{N-1}  \left(\omega^{k-1} (1+\sigma) - \omega^{k} \sigma\right) \mathcal{P}(x^{k})  \nonumber \\
	 &\ge \left(  \sigma + \sum_{k=0}^{N-1}  \omega^{k} \right)  \mathcal{D}\left(\frac{ \sigma x^{0} + \sum_{k=0}^{N-1}  \omega^{k} \bar{x}^{k+1} }{ \sigma + \sum_{k=0}^{N-1} \omega^{k} }\right) \nn \\
	&\quad  + \left(\sum_{k=0}^{N-1}  \omega^{k}\right) \mathcal{P}\left(\frac{\sum_{k=0}^{N-1}  \omega^{k} \tilde{y}^{k+1}}{\sum_{k=0}^{N-1}  \omega^{k}}\right)- \sigma  \mathcal{P}(x^{0})   \nonumber \\
	&\ge   \left( \sum_{k=1}^{N} \omega^{k-1} \right)  \left(\mathcal{P}\left(\frac{ \sigma x^{0} + \sum_{k=1}^{N} \omega^{k-1} \bar{x}^{k} }{  \sigma + \sum_{k=1}^{N} \omega^{k-1}} \right) +  \mathcal{D}\left(\frac{\sum_{k=1}^{N} \omega^{k-1} \tilde{y}^{k}}{\sum_{k=1}^{N} \omega^{k-1}}\right)\right) \nn \\
	&\quad -   \sigma  \mathcal{P}(x^{0})   \nonumber \\
	 &= \frac{\omega^{N} - 1}{\omega - 1}  \left(\mathcal{P}(\mathbf{x}^{N}) +  \mathcal{D}(\mathbf{y}^{N})\right) -   \sigma  \mathcal{P}(x^{0}).
	\end{align}
	By substituting \eqref{eq46} into \eqref{eq44}, we arrive at
	\begin{align}\label{eq:thm4.2-1}
	\mathcal{G}(\mathbf{x}^{N},\mathbf{y}^{N}) &= \mathcal{P}(\mathbf{x}^{N}) +  \mathcal{D}(\mathbf{y}^{N}) \nn \\
	& \le  \frac{\omega - 1}{\omega^{N} - 1 } \left(\tau \B_{\varphi}(\hat{y},y^{0})+\mu\B_{\psi}(\hat{x},x^{0}) +   \sigma  \mathcal{P}(x^{0})\right).
	\end{align}
	By utilizing the fact that $\omega > 1$, it is easy to see that
	\begin{equation*}
	\frac{\omega - 1}{\omega^{N} - 1} \le \frac{\omega}{\omega^{N} - 1} \le \frac{\omega}{\omega^{N} - \omega} \le \frac{1}{\omega^{N-1} - 1} \le \frac{1}{\omega^{N-2}},
	\end{equation*}
	which, together with \eqref{eq:thm4.2-1}, implies that
	\begin{equation*}
	\mathcal{G}(\mathbf{x}^{N},\mathbf{y}^{N})  \le  \frac{1}{\omega^{N-2}}\left(\tau \B_{\varphi}(\hat{y},y^{0})+\mu\B_{\psi}(\hat{x},x^{0}) +   \sigma  \mathcal{P}(x^{0})\right).
	\end{equation*}
	The proof is complete.
\end{proof}

\begin{remark}
	The inequality \eqref{eq27} means that Algorithm \ref{alg3} shares the same $O(1/\omega^N)$ linear convergence with the PDHG \cite{CP11} when the objective functions $f$ and $g$ are strongly convex relative to $\psi$ and $\varphi$, respectively.
\end{remark}

\section{Numerical experiments}\label{Sec:numexp}
To assess the reliability of our balanced approach in enhancing the efficiency of solving saddle point problems, in this section, we consider two classical convex optimization problems with readily solvable dual subproblems.  Specifically, we report the numerical performance of Algorithm \ref{alg1} (denoted by TBDA) on solving quadratic optimization problems and robust principal component analysis (RPCA). We implement all algorithms in Matlab 2023a, and all experiments are conducted on a 64-bit Windows PC with an Intel(R) Core(TM) i5-12500h CPU@3.000GHz and 32GB RAM.

\subsection{Quadratic optimization problem}\label{sec-qov}
Consider the following convex quadratic optimization problem:
\begin{equation}\label{QO}
\begin{aligned}
\min_{x } \;\; & \frac{1}{2}x^{\top} Q x + q^{\top} x  \\
{\rm s.t.}\;\; & Ax \leq b, \; x \ge 0,
\end{aligned}
\end{equation}
where $Q \in \R^{n \times n}$ is a positive definite matrix, $q \in \R^{n}$, $A \in \R^{m \times n}$, and $b \in \R^{m}$. 
Obviously, we can reformulate \eqref{QO} as a saddle point problem, i.e.,
\begin{equation*}
\min_{x \ge 0}\max_{y \ge 0} \left\{\; \L(x,y):= \frac{1}{2}x^{\top} Q x + q^{\top} x + \langle Ax,y \rangle -\langle b,y \rangle\; \right\}.
\end{equation*}
To generate the data of \eqref{QO}, we first let $Q = S^{\top}S + 2I_{n}$, where $S \in \R^{n \times n}$ is a random matrix generated by the Matlab script ``{\sffamily S=rand(n,n)}'' and $I_n$ is an $n\times n$ identity matrix. In this way, we can guarantee the positive definiteness of $Q$. Then, we randomly generate $A$ by the Matlab script ``{\sffamily A=rand(m,n)}''.  Now, we turn our attention to generating $q \in \R^{n}$ and $b \in \R^{m}$. Here, we first randomly generate two nonnegative sparse vector $x^\star$ and $y^\star$ by the Matlab scripts ``{\sffamily x\_star=sprand(n,1,0.4)}'' and ``{\sffamily y\_star=sprand(m,1,0.3)}'', respectively. Then, we let $q=-Qx^\star-A^\top y^\star$ and $b=Ax^\star +{\bm \varepsilon}^\diamond$, where ${\bm \varepsilon}^\diamond\in \mathbb{R}^m$ is also random nonnegative vectors satisfying  $\langle y^\star, {\bm \varepsilon}^\diamond\rangle =0$, respectively. By the first-order optimality condition, we can easily verify that $x^\star$ and $y^\star$ are optimal solutions to \eqref{QO}.

Note that the quadratic objective function of \eqref{QO} is strongly convex. We compare our TBDA and Algorithm \ref{alg2} (denoted by ITBDA) with the PDHG \cite{CP11} (see \eqref{pdhg}). For simplicity, we take the Bregman kernel functions for $x$- and $y$-subproblems as $\psi(x)=\frac{1}{2}\|x\|_{rI_{n}-\mu Q}^2$  and $\phi(y)=\varphi(y)=\frac{1}{2}\|y\|^2$. Thus, we set $r = \mu \lambda_{\max}(Q)+3$ and $\bar{\lambda} = \|rI_{n}-\mu Q\|$. In addition, there is a parameter $\theta$ in \eqref{c6}, we consider three choices $\theta =\{\frac{2}{3}, 1, 2\}$ for TBDA, and denote them by TBDA($\theta=\frac{2}{3}$), TBDA($\theta=1$) and TBDA($\theta=2$), respectively. For the algorithmic parameters under convergence-guaranteeing conditions, we set  $(\gamma,\mu,\sigma) = (\bar{\lambda}\|A\|,\|A\|, 1)$ for PDHG, $(\gamma,\mu,\tau,\sigma) = (4\bar{\lambda}\|A\|,\|A\|,\frac{8}{3}\bar{\lambda}\|A\|,1)$ for TBDA$(\theta = \frac{2}{3})$, $(\gamma,\mu,\tau,\sigma) = (\frac{3}{2}\bar{\lambda}\|A\|,\frac{8}{9}\|A\|,\frac{3}{2}\bar{\lambda}\|A\|,1)$ for TBDA($\theta = 1$), $(\gamma,\mu,\tau,\sigma) = (\frac{8}{7}\bar{\lambda}\|A\|,\frac{7}{9}\|A\|,\frac{16}{7}\bar{\lambda}\|A\|,1)$ for TBDA($\theta = 2$), and $(\gamma,\mu,\tau,\sigma)$ =  $(\bar{\lambda}\|A\|,\frac{2}{3}\|A\|,2\bar{\lambda}\|A\|,1)$ for ITBDA with $p=1.5$.

In our experiments, we consider nine scenarios on the problem's size, i.e., $(m,n) = (256i,512i)$ with $i = 2,\ldots,10$. Due to the randomness of data, we report the average results of $10$ trials for each scenario. On the other hand, since the optimal solutions are known, we employ
\begin{equation*}
\text{Tol}:=\frac{\|(x^{k+1},y^{k+1})-(x^{*},y^{*})\|}{\|(x^{*},y^{*})\|}\leq 10^{-6}
\end{equation*}
as the stopping criterion for all algorithms. We report the number of iterations (Iter) and computing time in seconds (Time) in Table \ref{tab1}.

\begin{figure}[htbp]
	\centering
	\includegraphics[width=0.495\linewidth]{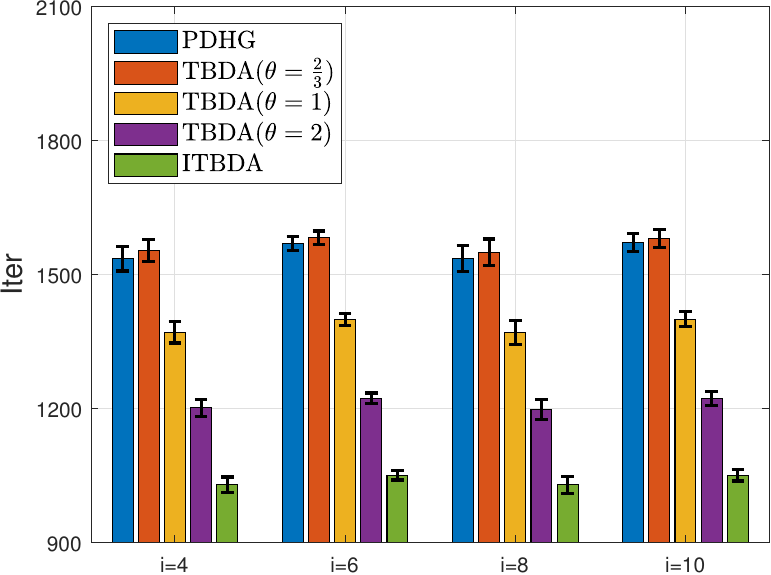}
	\includegraphics[width=0.48\linewidth]{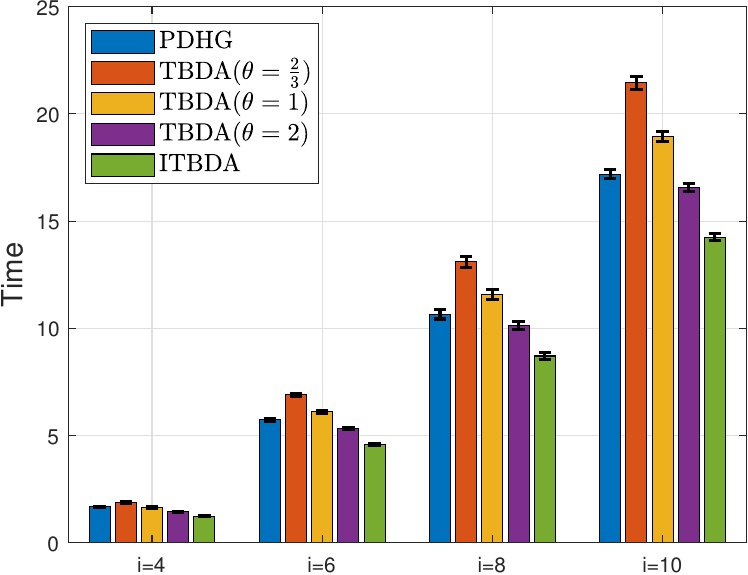}
	\caption{Numerical comparison between the PDHG and our algorithm for \eqref{QO}, where the 
		bars represent the average results of 10 random trials and the line segments represent the standard deviation showing the stability of the algorithms.}\label{Fig1}
\end{figure}

\begin{table}
	\centering
	\caption{Computational results for quadratic optimization problem.}\label{tab1}
	\scriptsize\begin{tabular}{c c c c c c c c c c c c c c c c c c}\toprule
		\multirow{2}{*}{\shortstack{  $i$}}&\multicolumn{2}{c}{PDHG}&&\multicolumn{2}{c}{TBDA($\theta = \frac{2}{3}$)}&&\multicolumn{2}{c}{TBDA($\theta = 1$)}&&\multicolumn{2}{c}{TBDA($\theta = 2$)}&&\multicolumn{2}{c}{ITBDA}  \\
		\cmidrule{2-3}\cmidrule{5-6}\cmidrule{8-9}\cmidrule{11-12}\cmidrule{14-15}\cmidrule{16-18}
		&Iter &Time &&Iter &Time &&Iter &Time &&Iter &Time &&Iter &Time  \\
		\hline
		2 & 1497.3 & 0.15&& 1521.0 & 0.17&& 1339.7 & 0.16&& 1177.0 & 0.14&& 1007.2 & 0.12 \\ 
		3& 1527.7 & 0.66&& 1539.6 & 0.74&& 1361.4 & 0.65&& 1189.5 & 0.57&& 1022.0 & 0.49 \\ 
		4& 1535.5 & 1.68&& 1554.1 & 1.87&& 1370.8 & 1.64&& 1201.8 & 1.44&& 1029.6 & 1.23 \\ 
		5& 1569.2 & 3.26&& 1580.5 & 3.77&& 1398.7 & 3.34&& 1222.0 & 2.91&& 1050.2 & 2.51 \\ 
		6& 1570.8 & 5.73&& 1582.8 & 6.91&& 1399.8 & 6.10&& 1223.6 & 5.32&& 1050.6 & 4.56 \\ 
		7& 1503.8 & 7.93&& 1514.5 & 9.87&& 1340.4 & 8.73&& 1169.7 & 7.63&& 1006.0 & 6.55 \\ 
		8& 1536.3 & 10.65&& 1550.3 & 13.11&& 1370.3 & 11.59&& 1198.0 & 10.13&& 1029.3 & 8.71 \\ 
		9& 1504.3 & 13.15&& 1518.7 & 16.35&& 1342.4 & 14.44&& 1173.7 & 12.68&& 1007.8 & 10.89 \\ 
		10& 1572.3 & 17.19&& 1581.5 & 21.46&& 1400.6 & 18.96&& 1222.9 & 16.57&& 1051.1 & 14.25 \\ 
		\toprule
	\end{tabular}
\end{table}	

We can see from Table \ref{tab1} that our proposed TBDA with $\theta=1,2$ takes fewer iterations than the PDHG. Moreover,  TBDA with $\theta=2$ and the improved version ITBDA always works better than the PDHG in terms of iterations and computing time. Due to the randomness of the data, we further show the average performance and the corresponding standard deviation of the iterations and computing time in Figure \ref{Fig1}, which shows the stability of all algorithms for solving quadratic programming problems. The plots in Figure \ref{Fig1} demonstrate that the TBDA with $\theta=2$ and ITBDA perform a little more stable than the PDHG. Particularly, the ITBDA works more efficient than the others when the objective function possesses the strongly convexity.

\subsection{RPCA}\label{sec-rpca}
In this subsection, we show the applicability of Algorithm \ref{alg1} to multi-block separable convex optimization problems as discussed in Section \ref{Sec:MALM}. More concretely, we consider the classical RPCA problem, which reads as
\begin{equation}\label{RPCA}
\min_{X,Z}\; \left\{\;\| X \|_{*} + \lambda \| Z \|_{1} \; |\; X+Z=H \; \right\},  
\end{equation}
where $\|X\|_{*}$ is the nuclear norm (i.e., the sum of the singular values of $X$) and $\|Z\|_{1}$ represents the sum of the absolute value of $Z$'s elements, and $H \in \mathbb{R}^{m \times n}$ is a given matrix. Obviously, the RPCA model  \eqref{RPCA} can be reformulated as the following separable saddle point problem:
\begin{equation}\label{P2}
\min_{X,Z}\max_{Y}\; \left\{\;\L(X,Z,Y)=\| X \|_{*} + \lambda \| Z \|_{1} + \langle X+Z,Y \rangle - \langle H,Y \rangle\;\right\}.
\end{equation}
Given the complexity of \eqref{P2}, in this subsection, we compare TBDA with PDHG and SPIDA while also examining the sensitivity of TBDA to its parameters. In our experiments, we consider some synthetic and real datasets, respectively.

We first conduct the numerical performance of TBDA on synthetic datasets. To generate the data, we let $X^{\star} = UV$, where $U\in \mathbb{R}^{m \times r}$ and $V \in \mathbb{R}^{r \times n}$ are independent random matrices whose entries come from Gaussian distribution $\mathcal{N}(0, 1)$. Clearly, the preset $r$ controls the rank of $X^{\star}$ to satisfy the low-rank property. Then, to generate the sparse matrix $Z^{\star}$, we first randomly select a support set $\Omega$ of size $0.15 \times mn$ (i.e., $15\%$ non-zero components) so that all elements are independently sampled from a uniform distribution in $[-30, 30]$. Finally, we let $H = X^{\star} + Z^{\star}$ be the observation matrix. In this way, we can guarantee that $(X^{\star} , Z^{\star})$ is the optimal solution of \eqref{RPCA}. Here, when implementing the SPIDA, we take the Bregman functions $\psi(\cdot)=\frac{1}{2}\|\cdot\|_F^2$ for the $X$- and $Z$-subproblems, and $\phi(Y)=\frac{1}{2}\|Y\|_F^2$ the $Y$-subproblem. In our experiments, we take the same step size for PDHG and SPIDA, i.e., $(\gamma_{0},\mu_{0}) = (\|A\|, \|A\|)$. To investigate the sensitivity of TBDA's step sizes, we first fix $\sigma = 1$ and $\tau = \gamma$, then TBDA with $(\gamma,\mu)$ = $(p_{1}\gamma_0,p_{2}\mu_0)$ is denoted by ``TBDA$(p_{1},p_{2})$" for simplicity, where we consider three groups of $(p_1,p_2)$ with $\{(0.91,0.91),(0.83,1.00),(1.00,0.83)\}$. Besides, our experiments employ the following stopping criterion
\begin{equation}\label{eb}
		\text{Tol}:= \frac{\|(X^{k+1},Z^{k+1},Y^{k+1})-(X^k,Z^k,Y^k)\|_F}{\|(X^k,Z^k,Y^k)\|_F} \leq \epsilon
\end{equation} and set $\epsilon = 10^{-5}$ for all algorithms. In the numerical simulations, we consider four scenarios on $(m,n)=(256i, 512i)$ and $r = 0.15*256i$ (by Matlab script ``{\sffamily r = round(0.15*min(m,n))}'') with $i=1,2,3,4$. All results are summarized in Table \ref{tab5}, where we report the rank of the recovered low-rank matrix $\hat{X}$ (denoted by $\mathrm{rank}(\hat{X})$), the number of nonzero components of the obtained sparse matrix $\hat{Z}$ (denoted by $\|\hat{Z}\|_0$), and the relative error defined as
\begin{equation*}
	\text{\sffamily Rerr} = \frac{\| \hat{X} + \hat{Z} -X^{\star} - Z^{\star}\|_{F}}{\| X^{\star}+Z^{\star}\|_{F}},
	\end{equation*}
the number of iterations (Iter.), and computing time in seconds (Time). It can be seen from Table \ref{tab5} that out TBDA takes fewer iterations than PDHG and SPIDA to achieve almost the same approximate solutions. Moreover, with increasing problem size, the superiority of our TBDA becomes more evident. The results in Table \ref{tab5} verify that our TBDA is efficient for structured convex optimization problems.
\begin{table}[htbp]
	\caption{Numerical results of RPCA with synthetic data sets.}\label{tab5}
	\centering
	\footnotesize{\begin{tabular*}{\textwidth}{@{\extracolsep{\fill}}ccccccc}
			\toprule
			$(m,n,r)$ & Methods & $\text{rank}(\hat{X})$ & $ \| \hat{Z} \|_{0} $ & $ \text{\sf Rerr}$ & Iter. & Time \\ \toprule
			\multirow{5}{*}{\shortstack{  $(256,512,38)$}}
			& PDHG & 38 & 22835 & 1.0063$\times 10^{-4}$ & 702 & 6.67 \\ 
			& SPIDA & 38 & 22893 & 1.0057$\times 10^{-4}$ & 636 & 6.03 \\ 
			& TBDA(0.91,0.91) & 38 & 23278 & 1.1413$\times 10^{-4}$ & 598 & 5.77 \\ 
			& TBDA(0.83,1.00) & 38 & 24015 & 2.1350$\times 10^{-4}$ & 485 & 4.68 \\ 
			& TBDA(1.00,0.83) & 38 & 22898 & 8.5263$\times 10^{-5}$ & 630 & 6.05 \\ 
			\midrule 
			\multirow{5}{*}{\shortstack{  $(512,1024,77)$}}
			& PDHG & 77 & 79598 & 3.2149$\times 10^{-5}$ & 594 & 28.85 \\ 
			& SPIDA & 77 & 79689 & 3.2118$\times 10^{-5}$ & 538 & 26.15 \\ 
			& TBDA(0.91,0.91) & 77 & 80400 & 5.6643$\times 10^{-5}$ & 387 & 19.03 \\ 
			& TBDA(0.83,1.00) & 77 & 80692 & 5.6550$\times 10^{-5}$ & 421 & 20.84 \\ 
			& TBDA(1.00,0.83) & 77 & 80114 & 3.5288$\times 10^{-5}$ & 469 & 23.19 \\ 
			\midrule 
			\multirow{5}{*}{\shortstack{  $(1024,2048,154)$}}
			& PDHG & 154 & 314800 & 1.2050$\times 10^{-5}$ & 436 & 125.58 \\ 
			& SPIDA & 154 & 314853 & 1.2152$\times 10^{-5}$ & 391 & 112.82 \\ 
			& TBDA(0.91,0.91) & 154 & 315802 & 1.8001$\times 10^{-5}$ & 322 & 93.58 \\ 
			& TBDA(0.83,1.00) & 154 & 316057 & 1.5845$\times 10^{-5}$ & 363 & 105.60 \\ 
			& TBDA(1.00,0.83) & 154 & 315591 & 1.2596$\times 10^{-5}$ & 357 & 103.70 \\ 
			\midrule 
			\multirow{5}{*}{\shortstack{  $(1536,3072,230)$}}
			& PDHG & 230 & 707849 & 5.2862$\times 10^{-6}$ & 523 & 461.15 \\ 
			& SPIDA & 230 & 707912 & 5.4040$\times 10^{-6}$ & 457 & 403.08 \\ 
			& TBDA(0.91,0.91) & 230 & 708742 & 1.0542$\times 10^{-5}$ & 333 & 295.51 \\ 
			& TBDA(0.83,1.00) & 230 & 708898 & 1.0702$\times 10^{-5}$ & 334 & 296.52 \\ 
			& TBDA(1.00,0.83) & 230 & 708674 & 9.2882$\times 10^{-6}$ & 347 & 308.05 \\ 
			\bottomrule 
	\end{tabular*}}
\end{table}

All datasets in above experiments are synthetic with known optimal solutions. Below, we conduct the numerical performance of TBDA on real datasets. Recalling that the model \eqref{RPCA} has been widely applied to the separation of background from surveillance videos. Therefore, we here consider four video sets (i.e., (i) hall, (ii) highway, (iii) hall airport, and (iv) lobby). For illustrating experiments, we only select the first $300$ frames of each video to construct the observation matrix $H\in\mathbb{R}^{m\times 300}$, where $m = m_{1} \times m_{2}$ with $m_{1}$ and $m_{2}$ being the video's height and width, respectively. Note that the true rank and sparsity of these videos are unknown. We set $\epsilon = 5 \times 10^{-5}$ in \eqref{eb} as the stopping criterion in real datasets. Here, we follow the parameters' settings used in the above synthetic datasets for PDHG and SPIDA. Then, we further investigate the influence of the extrapolated parameter $\sigma$ on TBDA. Therefore, we take $\gamma =  \frac{2(1+a)^{2}}{3+6a}\gamma_{0}$, $\mu = \mu_{0}$ and $\tau =  \frac{4(1+a)^{2}}{3+6a}\tau_{0}$, where $\sigma=a$. Correspondingly, we denote TBDA($\sigma = a$) with $a=\{(1,2,3)\}$. Unlike the results in Table \ref{tab5}, we alternatively report the objective value ({\sf Obj.}) and the relative error ({\sf Err.}) defined by
\begin{equation*}
\text{\sf Err} = \frac{\| \hat{X} + \hat{Z} - H\|_{F}}{\| H \|_{F}}.
\end{equation*}

\begin{table}[htbp]
	\caption{Numerical results of RPCA with real-world data sets.}\label{tab6}
	\centering
	\small{\begin{tabular*}{\textwidth}{@{\extracolsep{\fill}}ccccccc}
			\toprule
			$(m,n)$ & Methods & {\sf Obj.}& {\sf Err.} & Iter. & Time \\ \toprule
			\multirow{6}{*}{\shortstack{ hall \\  $(101376,300)$}}
			& PDHG & 4756.4 & 5.8748$\times 10^{-3}$& 169 & 180.83 \\ 
			& SPIDA & 4765.2 & 5.6628$\times 10^{-3}$& 157 & 169.26 \\ 
			& TBDA($\sigma = 1$) & 4423.7 & 1.3201$\times 10^{-2}$ & 97 & 107.82 \\ 
			& TBDA($\sigma = 2$) & 4368.4 & 1.4921$\times 10^{-2}$ & 94 & 104.67 \\ 
			& TBDA($\sigma = 3$) & 4328.4 & 1.6409$\times 10^{-2}$ & 93 & 102.93 \\ 
			\midrule 
			\multirow{6}{*}{\shortstack{ highway \\  $(101376,300)$}}
			& PDHG & 5413.9 & 5.3725$\times 10^{-3}$& 135 & 145.73 \\ 
			& SPIDA & 5420.1 & 5.2151$\times 10^{-3}$& 125 & 134.26 \\ 
			& TBDA($\sigma = 1$) & 5210.2 & 9.7609$\times 10^{-3}$& 105 & 117.27 \\ 
			& TBDA($\sigma = 2$) & 5129.8 & 1.1819$\times 10^{-2}$ & 106 & 118.39 \\ 
			& TBDA($\sigma = 3$) & 5037.4 & 1.4467$\times 10^{-2}$ & 99 & 110.29 \\ 
			\midrule 
			\multirow{6}{*}{\shortstack{ hall airport \\  $(25344,300)$}}
			& PDHG & 2895.3 & 7.3628$\times 10^{-3}$& 218 & 51.58 \\ 
			& SPIDA & 2905.0 & 6.8465$\times 10^{-3}$& 211 & 49.21 \\ 
			& TBDA($\sigma = 1$) & 2741.9 & 1.6136$\times 10^{-2}$ & 123 & 30.05 \\ 
			& TBDA($\sigma = 2$) & 2712.3 & 1.8427$\times 10^{-2}$ & 128 & 31.22 \\ 
			& TBDA($\sigma = 3$) & 2693.7 & 2.0004$\times 10^{-2}$ & 141 & 34.29 \\ 
			\midrule 
			\multirow{6}{*}{\shortstack{ lobby \\  $(20480,300)$}}
			& PDHG & 1363.5 & 9.5948$\times 10^{-3}$& 385 & 76.27 \\ 
			& SPIDA & 1367.6 & 9.2571$\times 10^{-3}$& 361 & 70.76 \\ 
			& TBDA($\sigma = 1$) & 1210.8 & 2.4023$\times 10^{-2}$ & 121 & 24.79 \\ 
			& TBDA($\sigma = 2$) & 1209.7 & 2.4572$\times 10^{-2}$ & 139 & 28.01 \\ 
			& TBDA($\sigma = 3$) & 1187.8 & 2.9030$\times 10^{-2}$ & 117 & 23.80 \\ 
			\bottomrule 
	\end{tabular*}}
\end{table}
Obviously, we can see from Table \ref{tab6} that our TBDA always runs faster than the PDHG and SPIDA for the real datasets. Moreover, the extrapolation parameter $\sigma$ is beneficial for accelerating the TBDA, which is important for saving computational cost when dealing with large-scale problems.  In Figure \ref{Fig6}, we list some backgrounds and foregrounds of the videos separated by the PDHG and TBDA. Clearly, these results show that our TBDA is reliable for dealing with the real datasets. 
\begin{figure}[htbp]
	\centering
	\includegraphics[width=0.19\linewidth]{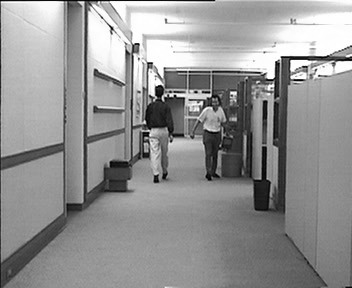}
	\includegraphics[width=0.19\linewidth]{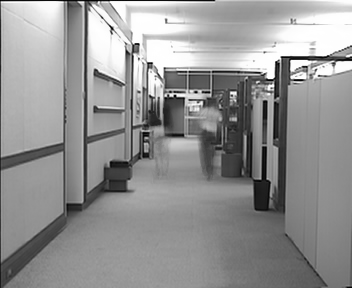}
	\includegraphics[width=0.19\linewidth]{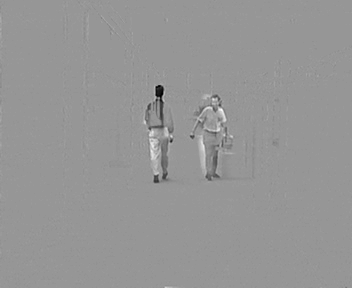}		
	\includegraphics[width=0.19\linewidth]{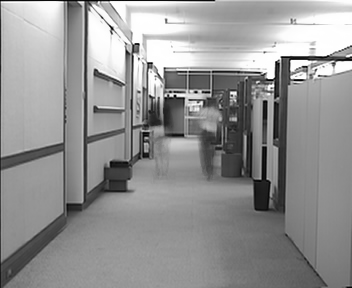}
	\includegraphics[width=0.19\linewidth]{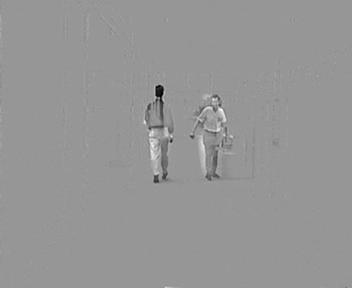}\\
		\includegraphics[width=0.19\linewidth]{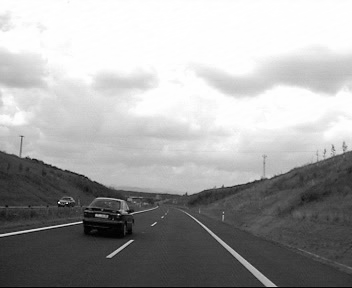}
	\includegraphics[width=0.19\linewidth]{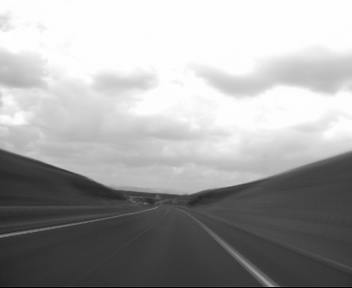}
	\includegraphics[width=0.19\linewidth]{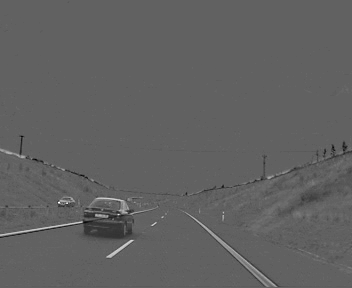}
	\includegraphics[width=0.19\linewidth]{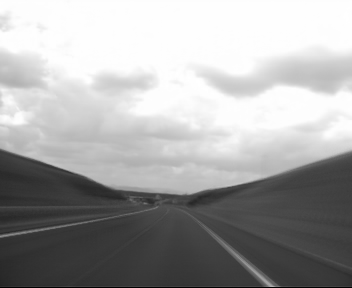}
	\includegraphics[width=0.19\linewidth]{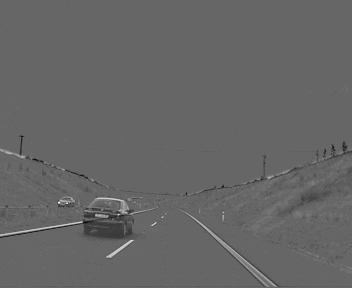}\\
	\includegraphics[width=0.19\linewidth]{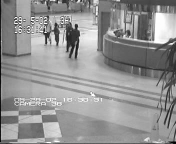}
	\includegraphics[width=0.19\linewidth]{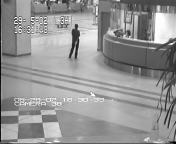}
	\includegraphics[width=0.19\linewidth]{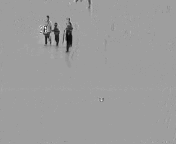}
	\includegraphics[width=0.19\linewidth]{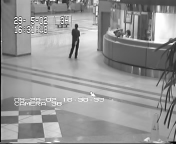}
	\includegraphics[width=0.19\linewidth]{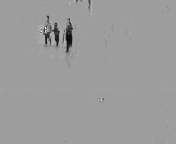}\\
	\includegraphics[width=0.19\linewidth]{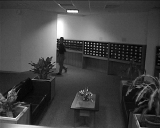}
	\includegraphics[width=0.19\linewidth]{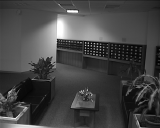}
	\includegraphics[width=0.19\linewidth]{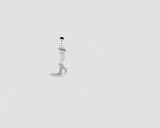}
	\includegraphics[width=0.19\linewidth]{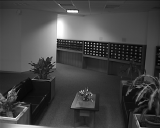}
	\includegraphics[width=0.19\linewidth]{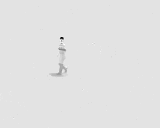}\\
	\caption{Results of the PDHG and TBDA for \eqref{RPCA} with application to real videos. From top to bottom: hall, highway, hall airport and lobby. The first column corresponds to one of the frames of each video. The second and third columns are the results obtained by the PDHG. The last two columns are the results obtained by our TBDA.}\label{Fig6}
\end{figure}

In summary, the results presented in this section demonstrate that our TBDA is both reliable and competitive for solving saddle point problems. Furthermore, they validate the core idea that incorporating an additional dual step and embedding distinct Bregman proximal terms into the subproblems can enhance the performance of primal-dual algorithms.

\section{Conclusions}\label{Sec:conclusion}
In this paper, we introduced a Triple-Bregman balanced primal-Dual Algorithm (TBDA) for solving a class of structured saddle point problems. The proposed TBDA offers a flexible and user-friendly algorithmic framework that not only help us understand the iterative schemes of some augmented Lagrangian-based algorithms, but also enables the design of customized variants with larger step sizes through appropriate choices of Bregman kernel functions and extrapolation parameters. Under relatively strong assumptions on the objective functions, we further developed two improved versions of TBDA. Computational results showed that TBDA outperforms the benchmark PDHG on both synthetic and real datasets, particularly when the dual subproblem is easier to solve than the primal one. 


\providecommand{\bysame}{\leavevmode\hbox to3em{\hrulefill}\thinspace}
\providecommand{\MR}{\relax\ifhmode\unskip\space\fi MR }
\providecommand{\MRhref}[2]{%
	\href{http://www.ams.org/mathscinet-getitem?mr=#1}{#2}
}
\providecommand{\href}[2]{#2}

\end{document}